\documentclass{amsart}
\usepackage{amssymb}

\newtheorem{theorem}{Theorem}[section]
\newtheorem{lemma}[theorem]{Lemma}
\newtheorem{corollary}[theorem]{Corollary}
\newtheorem{cor}[theorem]{Corollary}
\newtheorem{definition}[theorem]{Definition}

\theoremstyle{remark}
\newtheorem{remark}[theorem]{Remark}

\def\Im{{\rm Im}}
\newcommand{\les}{\lesssim}

\newcommand{\const}{\mbox{\rm const}}
\newcommand{\beeq}{\begin{equation}}
\newcommand{\eneq}{\end{equation}}

\newcommand{\supp}{\mbox{\rm supp}}

\newcommand{\eps}{{\varepsilon}}
\newcommand{\R}{{\mathbb R}}

\newcommand{\calL}{{\mathcal L}}
\newcommand{\calS}{{\mathcal S}}

\newcommand{\calM}{{\mathcal M}}

\newcommand{\calH}{{\mathcal H}}

\newcommand{\less}{\lesssim}
\newcommand{\gtr}{\gtrsim}
\newcommand\nind{\noindent}

\def\la{\langle}
\def\ra{\rangle}
\numberwithin{equation}{section}

\begin{document}

\title[Decay for the Wave and Schr\"odinger evolutions: part I]{Decay
for the wave and Schr\"odinger evolutions on manifolds with conical ends, Part~I}

\author{Wilhelm Schlag}
\address{University of Chicago, Department of Mathematics,
5734 South University Avenue, Chicago, IL 60637, U.S.A.}
\email{schlag@math.uchicago.edu}
\thanks{The first author was partly supported by the National
Science Foundation DMS-0617854.}

\author{Avy Soffer}
\address{Rutgers University, Department of Mathematics, 110 Freylinghuysen Road, Piscataway, NJ 08854, U.S.A.}
\email{soffer@math.rutgers.edu}
\thanks{The second author was partly supported by the National
Science Foundation DMS-0501043.}

\author{Wolfgang Staubach}
\address{Department of Mathematics,
Colin Maclaurin Building, Heriot-Watt University, Edinburgh, EH14 4AS}
\email{W.Staubach@hw.ac.uk}

\subjclass[2000]{35J10}


\begin{abstract}
Let $\Omega\subset \R^N$ be a compact imbedded Riemannian manifold of
dimension~$d\ge1$ and define the $(d+1)$-dimensional Riemannian manifold
$\calM:=\{(x,r(x)\omega)\::\: x\in\R,\, \omega\in\Omega\}$ with  $r>0$ and smooth, and the natural metric
$ds^2=(1+r'(x)^2)dx^2+r^2(x)ds_\Omega^2$.  We require that $\calM$ has
conical ends: $r(x)=|x| + O(x^{-1})$ as $x\to \pm\infty$. The Hamiltonian flow on such manifolds
always exhibits trapping.  Dispersive
estimates for the Schr\"odinger evolution $e^{it\Delta_\calM}$ and
the wave evolution $e^{it\sqrt{-\Delta_\calM}}$ are obtained for
data of the form $f(x,\omega)=Y_n(\omega) u(x)$ where $Y_n$ are eigenfunctions of~$\Delta_\Omega$.
This paper treats the case $d=1$, $Y_0=1$. In Part~II of this paper we provide details for all
cases $d+n>1$.  Our method combines two main ingredients:

\noindent (A) a detailed scattering analysis of Schr\"odinger operators
of the form $-\partial_\xi^2 + V(\xi)$ on the line where $V(\xi)$ has inverse square behavior at infinity

\noindent (B) estimation of oscillatory integrals  by (non)stationary phase.
\end{abstract}

\maketitle

\section{Introduction}

\noindent It is well-known that the free Schr\"odinger evolution
on $\R^{n+1}$ satisfies the dispersive bound
\begin{equation}
\label{eq:disp} \| e^{it\Delta} f\|_\infty \less |t|^{-\frac{n}{2}}
\|f\|_1
\end{equation}
where $\Delta$ denotes the Laplacean in $\R^n$.  Similarly, solutions to the wave equation
\[
\Box u =0,\quad u(0)=u_0,\; \partial_t u(0)= u_1
\]
in $\R^{n+1}$
satisfy
\begin{equation}\label{eq:wave_dispRn}
\begin{aligned}
\| u(t,\cdot)\|_\infty &\les t^{-\frac{n-1}{2}}\Big( \|u_0\|_{\dot W^{\frac{n+1}{2},1}} + \|u_1\|_{\dot W^{\frac{n-1}{2},1}}\Big) \\
\| u(t,\cdot)\|_\infty &\les t^{-\frac{n-1}{2}}\Big( \|u_0\|_{\dot B^{\frac{n+1}{2}}_{1,1}} + \|u_1\|_{\dot B^{\frac{n-1}{2}}_{1,1}}\Big) \\
\end{aligned}
\end{equation}
in odd and even dimensions, respectively.
Another instance of
such decay bounds are the global Strichartz estimates
\beeq\label{eq:strich} \| e^{it\Delta}
f\|_{L^{2+\frac{4}{n}}(\R^{n+1})} \less \|f\|_{L^2(\R^n)} \eneq
\nind and mixed-norm variants thereof as well as the corresponding versions for the wave equation.

In this paper we establish a
decay estimate (valid for all $t$), similar to \eqref{eq:disp}, for
the Schr\"odinger and wave evolution on a class of non-compact
manifolds which exhibit trapping of the Hamiltonian flow.
There has been much activity around establishing dispersive and
Strichartz estimates for more general operators, namely for
Schr\"odinger operators of the form $H=-\Delta+V$ with a decaying
potential $V$ or even more general perturbations. The seminal papers
here are Rauch\cite{Rau}, Jensen-Kato\cite{JK}, and Jorne\'e-Soffer-Sogge\cite{JSS}.
 We refer the reader to
the survey~\cite{Sch2} for more recent references in this area.

Around the same time as \cite{JSS}, Bourgain\cite{Bou} found
Strichartz estimates on the torus. This is remarkable, as compact
manifolds do not exhibit dispersion as in~\eqref{eq:disp} which was
always considered a key ingredient of the $T^*T$ argument leading to~\eqref{eq:strich}. The theme
of Strichartz estimates on manifolds (both local and global in time)
was then developed further in several important papers, see
Smith-Sogge\cite{SS}, Staffilani-Tataru\cite{ST},
Burq-Gerard-Tzvetkov\cite{BGT1}, \cite{BGT2},
Hassel-Tao-Wunsch\cite{HTW1}, \cite{HTW2}, Robbiano-Zuily\cite{RZ},
and Tataru\cite{T}. Gerard\cite{G} reviews some of the recent work
in this field.

A recurring theme in this area is the importance of periodic geodesics
for Strichartz estimates. In fact, it is well-known that the
presence of periodic geodesics  can lead to a loss of
derivatives in the Strichartz bounds. The intuition here is that
initial data that are highly localized around a periodic geodesic and
possess high momentum traveling around this geodesic will lead to
so-called meta-stable states in the Schr\"odinger evolution {\em provided the geodesic is stable} as
for example on spheres. Metastable states remain ``coherent'' for a long time, which amounts
to absence of dispersion during that time, see for example~\cite{G}
(in the classical approximation, dispersive estimates are governed
by the Newtonian scattering trajectories --- classically speaking,
periodic geodesics are states that do not scatter).

For this reason, many authors have imposed explicit non-trapping
conditions, see \cite{SS}, \cite{HTW1}, \cite{HTW2}, \cite{RodTao}.
The relevance of this condition lies with the construction of a
parametrix, which always involves solving for suitable
bi-characteristics. On manifolds these bi-characteristics are
governed by the geodesics flow in the co-tangent bundle - hence the
relevance of periodic geodesics.

There is a large body of work on the so-called Kato smoothing
estimates where this non-trapping condition also features
prominently, see for example Craig-Kappeler-Strauss\cite{CKS},
Doi\cite{Doi}, and Rodnianski-Tao\cite{RodTao}.

We now define the class of asymptotically conical
manifolds $\calM$ that we shall be working with.

\begin{definition}
\label{def:cones} Let $\Omega\subset \R^N$ with metric $ds_\Omega^2$
be a $d$-dimensional compact imbedded Riemannian manifold and define the $(d+1)$-dimensional manifold
\[
\calM:=\{(x,r(x)\omega)\:|\: x\in\R,\; \omega\in\Omega\},\quad
ds^2=r^2(x)ds_{\Omega}^2 +(1+r'(x)^2)dx^2
\]
where $r\in C^\infty(\R)$ and $\inf_x r(x)>0$. We say that there is
a {\em conical end} at the right (or left) if
 \begin{equation}\label{eq:cone} r(x)=|x|\,(1+ h(x)),\quad
  h^{(k)}(x) =O(x^{-2-k}) \quad \forall\; k\geq 0
  \end{equation}
as $x\to\infty \; (x\to-\infty)$.
\end{definition}

Of course we can consider cones with arbitrary opening angles here but this adds
nothing of substance. Furthemore, the regularity assumption can be relaxed to finitely many derivatives,
but we do not comment on this issue any further.
With $\Omega=S^1$ the manifold $\calM$ reduces to a  surface of
revolution
\begin{equation*}
\calS = \{(x,r(x)\cos\theta,r(x)\sin\theta ):-\infty< x<\infty
,0\leq\theta\leq 2 \pi\}
\end{equation*}
with the metric $ds^2=r^2(x)d\theta^2+(1+r'(x)^2)dx^2$. It has a
periodic geodesic at all local extrema of $r$.  An example of such a
manifold is given by the one-sheeted hyperboloid: $\Omega=S^1$ and
$r(x)=\sqrt{1+|x|^2}=:\langle x\rangle$.
If $d\ge2$, the entire Hamiltonian flow on~$\calM$ is trapped on the
set $(x_0,r(x_0)\Omega)$ when $r'(x_0)=0$.

In what follows, $\{Y_n,\mu_n\}_{n=0}^\infty $
denote the $L^2$-normalized eigenfunctions and eigenvalues, respectively,  of~$-\Delta_\Omega$.
In other words, $-\Delta_\Omega Y_n = \mu_n^2 Y_n$ where $0=\mu_0^2< \mu_1^2\le \mu_2^2\le\ldots$

\begin{theorem}\label{thm1}
Let $\calM$ be asymptotically conical at both ends in the sense of
Definition~\ref{def:cones} with $d\ge1$ arbitrary. Then for all $t>0$ and all $n\ge0$,
\begin{align} \Vert
e^{it\Delta_{\calM}}\, Y_nf\Vert_{L^{\infty}(\calM )} &\le C(n,\calM)\,
t^{-\frac{d+1}{2}}\Vert f\Vert_{L^1(\calM )} \label{eq:schr_d}\\
\Vert e^{\pm it\sqrt{-\Delta_{\calM}}}\, Y_nf\Vert_{L^{\infty}(\calM )} &\le C(n,\calM)\,
t^{-\frac{d}{2}}\Big(\Vert f'\Vert_{L^1(\calM )} + \Vert
f\Vert_{L^1(\calM)}  \Big) \label{eq:wave_d}
\end{align}
provided $f=f(x)$ does not depend on~$\omega$.
\end{theorem}

We remark that in the flat case, i.e., $r=\const=1$ the evolutions factor into those on~$\Omega$ and $\R$
and the dispersive rates are of course the same as on~$\R$. As for the wave equation, \eqref{eq:wave_d} gives
the natural estimate for $\cos(t\sqrt{-\Delta_\calM})$ --- the number of derivatives appearing on the right-hand
side agrees with that in~\eqref{eq:wave_dispRn} when $n=1$ since $f$ really sees the evolution  along
a one-dimensional generator the ``missing'' angular derivatives being hidden in $C(n,\calM)$. For $\frac{\sin(t\sqrt{-\Delta_{\calM}})}{\sqrt{-\Delta_{\calM}} }$ one can
prove the stronger bound which only requires~$L^1$ data, but we do not elaborate on this here.

In this paper we only prove the case $d=1, n=0$. In  Part~II we consider the general case. It turns out that the all cases subsumed in
$d+n>1$ follow very much the same scheme whereas $d+n=1$ has some separate features. This is to be expected, as for  $N=2$
the dispersive estimates for $-\Delta_{\R^N}+V$ are quite different from those in $\R^N$ with $N\ge3$, compare \cite{Sch} to~\cite{JSS}.
This is due to the logarithmic singularity of $(-\Delta_{\R^2}-z)^{-1}$ at $z=0$ as compared to the boundedness of the resolvent when $N\ge3$.
Not surprisingly, the logarithmic issues reappear in Part~I but not in Part~II of this series.

We now briefly describe the main ideas behind the proofs of Theorem~\ref{thm1}.
First, using arc-length coordinates $\xi$ on $\calM$ and after
multiplying by the weight $r^{\frac{d}{2}}(\xi)$, we reduce matters to
the Schr\"odinger operator
\[ \calH_{d,n}:=-\partial_\xi^2 + \frac{\mu_n^2}{r^2(\xi)} + V_1(\xi) =: -\partial_\xi^2 + V(\xi) \]
on $\R_\xi$.  Here $V_1(\xi)$ is a smooth potential that behaves like
$\frac{1}{4}d(d-2)\xi^{-2}$ as $\xi\to\pm\infty$. If $d=1, n=0$, then $V(\xi)\sim -\frac{1}{4\xi^2}$ as $\xi\to\infty$ (it is
therefore an attractive potential),
whereas for $d+n>1$ the potential $V$ becomes repulsive (in fact, very much so as $n$ and $d$ increase).  On the one hand, this
difference  accounts for the separate treatment of $d+n=1$ here as opposed to part~II. On the other hand, since
\[
V(\xi) = \big[2\mu_n^2 + d(d-2)/4\big] \xi^{-2} + O(\xi^{-3})\qquad \text{\ \ as\ \ }|\xi|\to\infty,
\]
with a positive leading term when $d+n>1$, it is reasonable that the cases $d+n>1$ can be treated simultaneously.

In order to prove
our theorems, we express the resolvent kernel as \[
(\calH_{d,n}-(\lambda^2+i0))^{-1}(\xi,\xi')= \frac{f_+(\xi,\lambda)
f_{-}(\xi',\lambda)}{W(\lambda)} \] when $\xi>\xi'$. Here $f_{\pm}$
are the usual Jost solutions for $\calH_{d,n}$ at energy $\lambda^2$:
\[
\calH_{d,n} f_{\pm}(\cdot,\lambda) = \lambda^2 f_{\pm}(\cdot,\lambda), \qquad f_{\pm}(\cdot,\lambda)
\sim e^{i\xi\lambda}\text{\ as\ }\xi\to\pm\infty
\]
and
\[
W(\lambda) = W(f_-(\cdot,\lambda), f_+(\cdot,\lambda))
\]
is their Wronskian.

Let us now briefly recall what is know about the existence of the Jost solutions and the asymptotic behavior
of $W(\lambda)$ for general operators $\calH=-\partial_\xi^2 + V$, see for example  Deift-Trubowitz~\cite{DT} for these elementary facts of scattering theory:
 for potentials $V(\xi)$
satisfying $\langle\xi\rangle V(\xi)\in L^1(\R)$ the Jost solutions
exist and are continuous in $\lambda\in\R$ (in fact, they are
continuous in $\lambda\ne0$ under the weaker condition $V\in L^1$). Moreover, $W(\lambda)\sim 2i\lambda$
as $\lambda\to\infty$ and either $W(0)\ne0$ or $W(\lambda)\sim c\lambda$ as $\lambda\to0$. The former case is
said to be {\em nonresonant} whereas the latter is {\em resonant}; it occurs exactly if there is globally bounded nonzero solution
to $\calH f=0$. In the nonresonant case, $f\sim 1$ as $\xi\to\infty$ then necessarily implies that $f(\xi)$ grows linearly in $\xi$
as $\xi\to-\infty$.

In the case of an inverse square potential the behavior of $f_{\pm}(\cdot,\lambda)$ and thus also of $W(\lambda)$ as $\lambda\to0$
is radically different. Assuming for simplicity that the leading order asymptotic behavior of $V(\xi)$ is the same as $\xi\to\pm\infty$ (as it
is here) we single out two possible scenarios which emerge from our analysis: first, suppose that
\[
V(\xi) = (\nu^2-\frac14)\xi^{-2} + O(\xi^{-3})\text{\ \ as\ \ }\xi\to\infty
\]
where $\nu>0$ (the case $\nu=0$ differing by logarithmic corrections). Then either
$W(\lambda)\sim c\lambda^{1-2\nu}$ or  $W(\lambda)\sim c \lambda^\sigma$ for some $\sigma<1-2\nu$ as $\lambda\to0$.
Loosely speaking, the former can be viewed as an analogue of the {\em nonresonant}
case from the usual scattering theory whereas the latter is the {\em resonant} case.
The resonant case is characterized by the existence of a nonzero solution $u$ of $\calH u=0$ with asymptotic behavior $\xi^{\frac12-\nu}$ as $\xi\to\infty$
and $c\, |\xi|^{\frac12-\nu}$ as $\xi\to-\infty$ where $c\ne0$.
Note that in the special case $\nu=\frac12$, which puts us back in the $\la \xi\ra V\in L^1$ scenario, this is exactly the standard characterization
of a zero energy resonance: there exists a nontrivial globally bounded zero energy solution.
In the resonant case one might expect $\sigma=1$, but our analysis does not yield that conclusion.

To conclude this introduction, let us recall the well-known heuristic principle that the behavior of the spectral measure close
to zero energy is the decisive fact for the long term behavior of any wave evolution. Indeed, with $E$ being the spectral resolution of $\calH_{d,n}$,
\[
e^{it\calH_{d,n}} = \int_0^\infty e^{it\lambda} E(d\lambda)
\]
Thus, decay of this Fourier transform as $t\to\infty$ is reflected most strongly by the behavior of $E(d\lambda)$ around $\lambda=0$.
This of course explains the importance of analyzing $W(\lambda)$ close to $\lambda=0$.

We now describe the proof method in more detail.

\section{The basic setup}
\label{sec:setup}

\nind The Laplace-Beltrami operator on $\calM$ where the base
$\Omega$ is of dimension $d\ge1$, is

\beeq\label{eq:LapBelt}
\Delta_{\calM}=\frac{1}{r^d(x)\sqrt{1+r'(x)^2}}\,
\partial_x\left(\frac{r^d(x)}{\sqrt{1+r'(x)^2}}
\partial_x\right)+\frac{1}{r^2(x)}\Delta_\Omega
\eneq We  switch to arclength parametrization. Thus, let
$$\xi (x)=\int_0^x\sqrt{1+r'(y)^2}\, dy.$$
\noindent Then (\ref{eq:LapBelt}) can be written as
\beeq\label{eq:LapBelt2}
\Delta_{\calM}=\frac{1}{r^d(\xi)}\partial_{\xi}(r^d(\xi
)\partial_{\xi})+\frac{1}{r^2(\xi )}\Delta_\Omega \eneq \noindent
where we have abused notation: $r(\xi )$ instead of $r(x(\xi ))$.
Setting $\rho (\xi ):=\frac{d}{2}\frac{\dot{r}(\xi )}{r(\xi )}$
yields \beeq\label{2.1}\Delta_{\calM}\,y(\xi ,\omega )=
\partial^2_{\xi}y+2\rho\partial_{\xi}y+\frac{1}{r^2}\Delta_\Omega y.\eneq
We remove the first order term in \eqref{2.1} by setting
\beeq\label{3} y(\xi ,\omega )=r(\xi )^{-\frac{d}{2}}u(\xi ,\omega).
\eneq
 Then \beeq\label{eq:4} \Delta_\calM y =
\partial^2_{\xi}y+2\rho\partial_{\xi}y+\frac{1}{r^2}\Delta_{\Omega} y=r^{-d/2}[-\calH u+\frac{1}{r^2}\Delta_\Omega u]
\eneq \noindent with \beeq\label{eq:calH}\calH=-\partial^2_{\xi}+V,\quad  V(\xi )=\rho^2(\xi
)+\dot{\rho}(\xi ). \eneq Note that
the Schr\"odinger operator $\calH$ can be factorized as
\begin{equation}\label{eq:calL}
\calH =\calL^*\calL,\quad \calL = -\frac{d}{d\xi} + \rho
\end{equation}
In particular, $\calH$ has no negative spectrum. In terms of the
Schr\"odinger evolution,
\[
e^{-it\Delta_\calM}f = r^{-\frac{d}{2}} e^{it \calH} r^{\frac{d}{2}}
f \quad \forall\; f=f(\xi)
\]
and the same for the wave equation. In particular, any estimate of
the form
\[
\big\| e^{-it\Delta_\calM} f\|_{L^\infty(\calM)} \le Ct^{-\alpha}
\|f\|_{L^1(\calM)} \quad \forall\; t>0,\; f=f(\xi)
\] with arbitrary $\alpha\ge0$ and some constant $C$ that does not depend on~$t$,
is equivalent to one of the form
 \beeq\label{11}
\big\|r^{-\frac{d}{2}}e^{it\calH}r^{-\frac{d}{2}}u\big\|_{L^{\infty}(\R
)} \le C'\, t^{-\alpha} \|u\|_{L^1(\R)} \quad \forall\; t>0,\;
u=u(\xi)\eneq with a possibly different constant $C'$. Here we
absorbed the weight from the volume element $dv_\calM = r^d d\xi
dv_\Omega$ arising in the $L^1(\calM)$ norm into the left-hand side
of~\eqref{11}. An analogous reduction is of course valid for the
wave evolution.  As usual, the functional calculus applied
to~\eqref{11} yields
\[
e^{it\calH} = \int_0^\infty e^{it\lambda} E(d\lambda)
\]
where $E(d\lambda)$ is the spectral resolution of~$\calH$. The point
is that there is an ``explicit expression'' for~$E(d\lambda)$:
\[
E(d\lambda^2)(\xi,\xi') = 2\lambda \Big\{
\Im\Big[\frac{f_+(\xi,\lambda)f_-(\xi',\lambda )}{W(\lambda
)}\Big]\chi_{[\xi>\xi']} + \Im
\Big[\frac{f_-(\xi,\lambda)f_+(\xi',\lambda )}{W(\lambda )}\Big]
\chi_{[\xi<\xi']} \Big\}\,d\lambda
\]
where \[ W(\lambda ):=W(f_-(\cdot ,\lambda),f_+(\cdot ,\lambda))=
f_+'(\cdot ,\lambda) f_-(\cdot ,\lambda) - f_-'(\cdot ,\lambda)
f_+(\cdot ,\lambda)
\] is the Wronskian of the solutions $f_{\pm}(\cdot ,\lambda )$ of
the following ordinary differential equation
 \beeq
\begin{aligned}\label{13}
\calH f_{\pm}(\xi ,\lambda )&=-f_{\pm}''(\xi ,\lambda )+V(\xi
)f_{\pm}(\xi ,\lambda ) =\lambda^2\, f_{\pm}(\xi ,\lambda)\\
f_{\pm}(\xi ,\lambda ) &\sim e^{\pm i\lambda\xi}\qquad \mathrm{as}\;
\xi\to\pm\infty
\end{aligned}
\eneq
 \noindent provided $\lambda\neq 0$. The functions $f_{\pm}$
are called the {\em Jost solutions} and it is a standard fact that
these solutions exist because of the decay of~$V$ which turns out to
be
$$|V(\xi )|\les\langle\xi\rangle^{-2}.$$  To establish this, as well
as an important refinement thereof, we start with the following
elementary consequence of Definition~\ref{def:cones}.
\begin{definition}
In what
follows, a term $O(x^{-\gamma})$ is said to {\em behave like a
symbol} if $|\partial_x^\ell O(x^{-\gamma})|\les x^{-\gamma-\ell}$
as $x\to\infty$ for all $\ell\ge1$.
\end{definition}
Furthermore, we shall assume
henceforth that both ends of~$\calM$ are conical, i.e.,
\eqref{eq:cone} holds.

\begin{lemma}
  \label{lem:basic_asymp} With suitable constants  $c_{\infty},\tilde
c_\infty$, and  as $x\to \infty$ \beeq\label{7} \xi (x)=
\sqrt{2}\,x+c_{\infty}+O(x^{-1})
 \eneq
as well as  \beeq\label{29} r(\xi )=
\frac{1}{\sqrt{2}\,}\xi
\left(1-\frac{c_{\infty}}{\xi}+O(\xi^{-2})\right)
 \eneq  as $\xi\to\infty$.
Moreover, the $O$-terms behave like symbols.
\end{lemma}
\begin{proof}
We plug $r(x)=x(1+O(x^{-2}))$ and thus $r'(x)=1+O(x^{-2})$ into the
expression for~$\xi$, i.e.,
\begin{align*}
\xi(x)&=\int_0^x \sqrt{2+O(\la y\ra^{-2})}\, dy = \sqrt{2}\, x +
\int_0^x O(\la y\ra^{-2})\, dy \\
&= \sqrt{2}\, x + \int_0^\infty  O(\la y\ra^{-2})\, dy + O(x^{-1}) =
\sqrt{2}\, x + c_\infty + O(x^{-1})
\end{align*}
Hence,
\[
r(x) = x + O(x^{-1}) = 2^{-\frac12}\, (\xi-c_\infty) + O(\xi^{-1})
\]
as claimed.  The symbol behavior follows from the fact that the
errors in Definition~\ref{def:cones} also behave like symbols.
\end{proof}

As a corollary, we obtain

\begin{cor}
  \label{cor:Vdec} The potential $V$ from~\eqref{eq:calH} has the form
 \beeq\label{14} V(\xi
)= \Big( \frac{d^2}{4} - \frac{d}{2}\Big) \xi^{-2}
+O(\xi^{-3})\qquad \mathrm{as}\; \xi \to\infty \eneq where
$O(\xi^{-3})$  behaves
 like a symbol.
\end{cor}
\begin{proof}
  Simply observe that at a conical end, $\rho=\frac{d}{2}\frac{\dot r}{r}= \frac{d}{2}\xi^{-1}(1+O(\xi^{-1}))$
  as $\xi\to\infty$. Hence,
  \[
V(\xi)=\dot\rho(\xi)+\rho^2(\xi) = \frac{1}{4}d(d-2)\xi^{-2} + O(\xi^{-3}) \text{\
\ as\ \ }\xi\to\infty
  \]
  as claimed. The behavior of the $O(\cdot)$ term follows from the
  fact that the $O(\cdot)$ in Lemma~\ref{lem:basic_asymp} are of
  symbol type.
\end{proof}

From \eqref{13}, $f_{\pm}(\cdot ,\lambda)$ are solutions of the {\em
Volterra integral equations} \beeq\label{eq:Volt} f_+(\xi ,\lambda
)=e^{i\lambda\xi}+\int_{\xi}^{\infty}\frac{\sin(\lambda (\eta
-\xi))}{\lambda}V(\eta )f_+(\eta ,\lambda )\, d\eta \eneq
 \noindent
and similarly for $f_-$.  For the convenience of the reader, we now
recall how to solve Volterra integral equations in general. Thus,
consider
\begin{equation*}
(\ast )\,\,\,f(x)= g(x)+\int_{x}^{\infty} K(x,s) f(s) ds,
\end{equation*}
or
\begin{equation*}
(\ast\ast )\,\,\,f(x)= g(x)+\int_{a}^{x} K(x,s) f(s) ds,
\end{equation*}
with some $g(x)\in L^{\infty}$ and $a\in \mathbb{R}$. As usual, one
solves them by an iteration procedure which requires finding a
suitable convergent majorant for the resulting series expansion.

\begin{lemma}\label{4'}
Let $a\in \mathbb{R}$ and $g(x)\in L^{\infty}(a,\infty)$. Let
\[\mu:= \int_{a}^\infty \sup_{a<x<s} |K(x,s)|\, ds<\infty\] Then
there exists a unique solution to $(\ast)$ given by
\begin{equation}
\label{eq:volt_it} f(x) = g(x) + \sum_{n=1}^\infty \int_a^\infty
\ldots \int_a^\infty \prod_{i=1}^n \chi_{[x_{i-1}<x_i]}
K(x_{i-1},x_i) \; g(x_{n}) \, dx_{n}\ldots dx_1.
\end{equation}
with $x_0:=x$.  Furthermore, one has the bound
\[ \|f\|_{L^\infty(a,\infty)} \le e^\mu \|g\|_{L^\infty(a,\infty)}, \]
and an analogue statement holds for $(\ast\ast)$.
\end{lemma}
\begin{proof}
We only prove the lemma for $(\ast)$ since the proof for
$(\ast\ast)$ is almost identical. The idea is simply to show that
the infinite Volterra iteration \eqref{eq:volt_it} for $(\ast)$
converges. To this end, define
\[ K_0(s) := \sup_{a<x<s} |K(x,s)| \]
Then
\begin{align*}
& \Big| \int_a^\infty \ldots \int_a^\infty
\prod_{i=1}^n \chi_{[x_{i-1}<x_i]} K(x_{i-1},x_i) \; g(x_{n}) \, dx_{n}\ldots dx_1 \Big| \\
& \le \int_a^\infty \ldots \int_a^\infty
\prod_{i=1}^n \chi_{[x_{i-1}<x_i]} K_0(x_i) \; |g(x_{n})| \, dx_{n}\ldots dx_1  \\
& = \|g\|_{L^\infty(a,\infty)}\frac{1}{n!} \int_a^\infty \ldots
\int_a^\infty
\prod_{i=1}^n  K_0(x_i) \, dx_{n}\ldots dx_1  \\
&= \frac{1}{n!} \|g\|_{L^\infty(a,\infty)} \Big(\int_a^\infty
K_0(s)\, ds\Big)^n
\end{align*}
Hence, the series in \eqref{eq:volt_it} converges absolutely and
uniformly in $x>a$ with the uniform upper bound
\[ \|g\|_{L^\infty(a,\infty)}\sum_{n=0}^\infty \frac{1}{n!}\mu^n = e^\mu\|g\|_{L^\infty(a,\infty)}\]
as claimed.
\end{proof}

It is now clear that~\eqref{eq:Volt} admits a solution for every
$\lambda\ne0$. At $\lambda=0$, we need to replace~\eqref{eq:Volt}
with
\[
f_+(\xi ,0 )=1+\int_{\xi}^{\infty}(\eta -\xi)V(\eta )f_+(\eta ,0)\,
d\eta
\]
 If $d\ne2$, then this integral equation has no
meaning due to the $\eta^{-2}$ decay of~$V(\eta)$, see~\eqref{14}.
Moreover, the zero energy solutions of $\calH u=0$ are given by
\begin{equation}\label{eq:u0u1}\begin{aligned}
u_0(\xi )&=r^{\frac{d}{2}}(\xi ),\\
u_1(\xi )&=r^{\frac{d}{2}}(\xi )\int_0^{\xi}r^{-d}(\eta)\, d\eta,
\end{aligned}
\end{equation}
\noindent see (\ref{eq:ydef}) and (\ref{3}).  Since no linear
combination of these functions can be made asymptotically constant
when $d\ne2$, it follows that (\ref{13}) itself has
no meaning at $\lambda=0$. Note, however,  that for $d=2$
\[r(\xi)\int_\xi^\infty r^{-2}(\eta)\,d\eta\] is asymptotically
constant at a conical end as $\xi\to\infty$ which is in agreement with
the fact that for $d=2$ the potential $V$ decays like an inverse
cubic.

In view of this discussion, we have reduced the decay estimates for
the Schr\"odinger equation to the following
  oscillatory
integral bounds:
 \beeq
\begin{aligned}\label{eq:schr_oscill}
& \sup_{\xi >\xi '} \, r^{-\frac{d}{2}}(\xi)
r^{-\frac{d}{2}}(\xi')\bigg|\int_0^{\infty}e^{it\lambda^2}
\lambda\, \text{Im}\left[\frac{f_+(\xi,\lambda)f_-(\xi',\lambda )}{W(\lambda )}\right]\, d\lambda \bigg|\\
&+\sup_{\xi <\xi '}\, r^{-\frac{d}{2}}(\xi)
r^{-\frac{d}{2}}(\xi')\bigg|\int_0^{\infty}e^{it\lambda^2} \lambda\,
\text{Im}\left[\frac{f_+(\xi',\lambda)f_-(\xi,\lambda )}{W(\lambda
)}\right]\, d\lambda \bigg|\les t^{-(d+1)/2}
\end{aligned}
\eneq For the wave-equation, the reduction takes the form
 \beeq
\begin{aligned}\label{eq:wave_oscill}
& \Big|\int_{-\infty}^\xi r^{-\frac{d}{2}}(\xi)
r^{-\frac{d}{2}}(\xi') \int_0^{\infty} e^{it\lambda}\,\lambda
\, \text{Im}\left[\frac{f_+(\xi,\lambda)f_-(\xi',\lambda )}{W(\lambda )}\right]\, d\lambda\, \phi(\xi')\,d\xi'\Big|\\
&+ \Big|\int_\xi^\infty r^{-\frac{d}{2}}(\xi)
r^{-\frac{d}{2}}(\xi')\int_0^{\infty} e^{it\lambda}\, \lambda
\text{Im}\left[\frac{f_+(\xi',\lambda)f_-(\xi,\lambda )}{W(\lambda
)}\right]\, d\lambda\, \phi(\xi')\, d\xi' \Big|\\
&\les t^{-d/2} \int (|\phi'(\eta)|+|\phi(\eta)|)\, d\eta
\end{aligned}
\eneq
uniformly in $\xi$.

\section{The scattering theory for $d=1, n=0$}
\label{sec:scat1}

The goal of this section is to obtain a sufficiently accurate
representation of $f_{\pm}(\cdot,\lambda)$ in~\eqref{eq:schr_oscill}
and~\eqref{eq:wave_oscill}.  We remark that using
\eqref{eq:LapBelt2}, one obtains two $\omega$ independent harmonic
functions on $\calM$:
\begin{equation}\label{eq:ydef}
y_0(\xi )=1,\quad y_1(\xi )=\int_0^{\xi}r^{-1}(\xi ')\, d\xi '
\end{equation}
At a conical end, $y_1(\xi)=\sqrt{2}\log \xi +O(1)$,
cf.~Lemma~\ref{lem:basic_asymp}. The related functions
$u_0=r^{\frac12}$ and $u_1=r^{\frac12} y_1$ from~\eqref{eq:u0u1} are
zero-energy solutions of $\calH$, see~\eqref{eq:calH}
and~\eqref{eq:calL}. Their asymptotics are as follows (assuming
throughout that $\calM$ is conical at the ends):

\begin{lemma}\label{lemma7}
As $\xi\to\infty$,
\begin{equation}\begin{aligned}\label{eq:u0u1conic}
u_0(\xi )&=2^{-1/4}\xi^{1/2}\Big(1-\frac{c_{\infty}}{2\xi}+O(\xi^{-2})\Big)\\
u_1(\xi
)&=2^{1/4}\xi^{1/2}\Big(1-\frac{c_{\infty}}{2\xi}+O(\xi^{-2})\Big)\Big(\log\xi+c_2+O(\xi^{-1})\Big).
\end{aligned}
\end{equation}
Here $c_2$ is some constant and the $O$-terms behave like symbols
under differentiation in~$\xi$.
\end{lemma}

\begin{proof} The expressions for $u_0$ are an immediate consequence
of Lemma~\ref{lem:basic_asymp}.  Simply compute
\begin{align*}
\int_0^{\xi}r^{-1}(\eta )\, d\eta &=
\int_0^{\xi}\sqrt{2}\,\langle\eta\rangle^{-1}\big(1+c_{\infty}\langle\eta\rangle^{-1}
+O(\langle\eta\rangle^{-2})\big)\, d\eta\\
&=\sqrt{2}\,(\log\xi +c_2)+O(\xi^{-1})\qquad\mathrm{as}\;
\xi\to\infty.
\end{align*}
Thus,
\begin{align*}
u_1(\xi )&=\sqrt{r(\xi )}\int_0^{\xi}r^{-1}(\eta )\, d\eta\\
&=2^{1/4}\xi^{1/2}\left(1-\frac{c_{\infty}}{2\xi}+O(\xi^{-2})\right)\left(\log\xi+c_2+O(\xi^{-1})\right)\qquad\mathrm{as}\;
\xi\to\infty.
\end{align*}
To symbol character of the $O(\cdot)$ terms here follows from the
fact that it was assumed in Definition~\ref{def:cones}.
\end{proof}

We now perturb the zero energy solutions relative to the energy. For
small energies and in the region $|\xi\lambda|\ll 1$, this produces
a useful approximation to the exact solutions.

\begin{lemma}\label{lemma1}
For any $\lambda\in\R$, define \beeq\label{18} u_j(\xi ,\lambda
):=u_j(\xi
)+\lambda^2\int_0^{\xi}[u_1(\xi)u_0(\eta)-u_1(\eta)u_0(\xi)]u_j(\eta
,\lambda )\, d\eta \eneq \noindent where $j=0,1$.  Then $\calH
u_j(\cdot ,\lambda )=\lambda^2u_j(\cdot ,\lambda )$ with $u_j(\cdot
,0)=u_j(\cdot )$, for $j=0,1$ and \beeq\label{19} W(u_0(\cdot
,\lambda ),u_1(\cdot ,\lambda ))=1 \eneq for all $\lambda$.
\end{lemma}

\begin{proof}
First, one checks that $W(u_0,u_1)=1$.  This yields $\calH u_j(\cdot
,\lambda )=\lambda^2u_j(\cdot ,\lambda )$ since $\calH u_j=0$ for
$j=0,1$.  Second, $u_j(0,\lambda )=u_j(0)$ and $u_j'(0,\lambda
)=u_j'(0)$ for $j=0,1$.  Hence $W(u_0(\cdot ,\lambda ),u_1(\cdot
,\lambda ))=u_1'(0)u_0(0)-u_1(0)u_0'(0)=1$.
\end{proof}

As an immediate corollary we have the following statement.

\begin{corollary}\label{cor1}
There exist $a_+(\lambda )$, $a_{-}(\lambda)$, $b_+(\lambda )$ and $b_-(\lambda
)$ such that with $f_{\pm}(\cdot ,\lambda )$ as in $(\ref{13})$, one has for any
$\lambda\neq 0$
\begin{equation}
\begin{aligned}\label{20}
f_+(\xi ,\lambda )&=a_+(\lambda )u_0(\xi ,\lambda )+b_+(\lambda )u_1(\xi ,\lambda )\\
f_-(\xi ,\lambda )&=a_{-}(\lambda)u_0(\xi ,\lambda )+b_-(\lambda)u_1(\xi ,\lambda ).
\end{aligned}
\end{equation} \noindent Furthermore $a_{\pm}(\lambda )=W(f_{\pm}(\cdot ,\lambda
),u_1(\cdot ,\lambda ))$, $b_{\pm}(\lambda )=-W(f_{\pm}(\cdot
,\lambda ),u_0(\cdot ,\lambda ))$, and
\begin{equation}\label{eq:Wagain}
  W(\lambda):= W(f_{-}(\cdot,\lambda),f_{+}(\cdot,\lambda))=
a_-(\lambda) b_+(\lambda) - a_+(\lambda) b_-(\lambda).
\end{equation}
Moreover, if $\calM$ is symmetric, then $a_-(\lambda )=a_+(\lambda
)$ and $b_-(\lambda )=-b_+(\lambda )$.
\end{corollary}

\begin{proof}
The Wronskian relations for $a_{\pm}$, $b_{\pm}$ follow immediately
from (\ref{19}).  The formula for $W(\lambda)$ also follows by
plugging~\eqref{20} into~\eqref{eq:Wagain}. In the symmetric case,
i.e., assuming $r(x)=r(-x)$ one also has $r(\xi )=r(-\xi )$.  In
particular, this implies that $f_-(-\xi ,\lambda )=f_+(\xi ,\lambda
)$ and $u_0(-\xi )=u_0(\xi )$ as well as $u_1(-\xi )=-u_1(\xi )$.
Thus,
\begin{align*} a_-(\lambda ) &= W(f_{-}(\cdot, \lambda),u_1(\cdot
,\lambda))=-W(f_-(-\cdot ,\lambda),u_1(-\cdot
,\lambda))\\
&=W(f_+(\cdot ,\lambda ),u_1(\cdot ,\lambda ))=a_+(\lambda )\\
 b_-(\lambda ) &=-W(f_-(\cdot ,\lambda),u_0(\cdot ,\lambda))=
W(f_-(-\cdot ,\lambda),u_0(-\cdot ,\lambda))\\
&=W(f_+(\cdot ,\lambda ),u_0(\cdot ,\lambda ))=-b_+(\lambda )
\end{align*}
as claimed.
\end{proof}

\subsection{The analysis of $f_+(\cdot,\lambda)$ at a conical end, $d=1$}

By Corollary~\ref{cor:Vdec}, \beeq\label{14'} V(\xi
)=-\frac{1}{4\xi^2}+V_1(\xi ),\qquad \xi\to\infty \eneq where
$|V_1(\xi )|\les |\xi |^{-3}$. Moreover, $\vert V_1^{(k)}(\xi
)\vert\les |\xi |^{-3-k}$ for $\xi >1$.

\begin{lemma}\label{lemma4}  Let
\[\calH_0:=-\partial^2_{\xi}-\frac{1}{4\xi^2}\]
For any $\lambda >0$ the problem
\begin{align*}
\calH_0f_0(\cdot ,\lambda )&=\lambda^2f_0(\cdot ,\lambda ),\\
f_0(\xi ,\lambda )&\sim e^{i\xi\lambda} \text{\ \ as\ \ }
\xi\to\infty
\end{align*}
has a unique solution on $\xi >0$.  It is given by \beeq\label{23}
f_0(\xi ,\lambda )=\sqrt{\frac{\pi}{2}}\, e^{i\pi
/4}\sqrt{\xi\lambda}\,H^{(+)}_0(\xi\lambda). \eneq Here
$H^{(+)}_0(z)=J_0(z)+iY_0(z)$ is the Hankel function of order zero.
\end{lemma}

\begin{proof}
It is well-known, see Abramowitz-Stegun\cite{AS}, that the ordinary
differential equation
$$w''(z)+\left(\lambda^2+\frac{1}{4z^2}\right)W(z)=0$$
has a fundamental system of solutions $\sqrt{z}\, J_0(\lambda z)$,
$\sqrt{z}\, Y_0(\lambda z)$ or equivalently, \[\sqrt{z}\,
H^{(+)}_0(\lambda z),\qquad \sqrt{z}\, H^{(-)}_0(\lambda z).\]
Recall the asymptotic relations
\begin{align*}
H^{(+)}_0(x)&\sim\sqrt{\frac{2}{\pi x}}\, e^{i(x-\frac{\pi}{4})}
\qquad\mathrm{as}\; x\to +\infty\\
H^{(-)}_0(x)&\sim\sqrt{\frac{2}{\pi x}}\,
e^{-i(x-\frac{\pi}{4})}\qquad\mathrm{as}\; x\to +\infty.
\end{align*}
Thus, (\ref{23}) is the unique solution so that $$f_0(\xi ,\lambda
)\sim e^{i\xi\lambda},$$ as claimed.
\end{proof}

Having these tools at our disposal, we proceed with our
investigation of the Jost solutions. To this end, instead of the
Volterra equation (\ref{eq:Volt}) we will work with the following
representation of the solutions of (\ref{13}):

\begin{lemma}\label{lemma5}
 For any $\xi >0$, $\lambda >0$, \beeq\label{24} f_+(\xi
,\lambda )=f_0(\xi ,\lambda )+\int_{\xi}^{\infty}G_0(\xi ,\eta
;\lambda )V_1(\eta )f_+(\eta,\lambda)\, d\eta \eneq with $V_1$ as in
$(\ref{14'})$, $f_0$ as in $(\ref{23})$ and \beeq\label{25} G_0(\xi
,\eta ;\lambda)=[\overline{f_0(\xi ,\lambda )}f_0(\eta, \lambda)
)-f_0(\xi,\lambda)\overline{f_0(\eta,\lambda)}](2i\lambda )^{-1}.
\eneq For any small $\lambda>0$ and $1<\xi<\lambda^{-1}$,
\begin{equation}\label{eq:G0_bd0}
|G_0(\xi ,\eta ;\lambda)|\les
(\xi\eta)^{\frac12}|\log\lambda|^2\chi_{[\xi<\eta<\lambda^{-1}]} +
(\xi/\lambda)^{\frac12}|\log\lambda|\chi_{[\eta>\lambda^{-1}]}
\end{equation}
\end{lemma}

\begin{proof}
Simply observe that $G_0$ is the Green's function of our problem
relative to $\calH_0$.  Indeed,
\begin{align*}
G_0(\xi ,\xi;\lambda )&=0,\\
\partial_{\xi}G_0(\xi,\eta;\lambda)|_{\eta =\xi}&=1,\\
\calH_0G_0(\cdot,\eta;\lambda )&=\lambda^2G_0(\cdot,\eta;\lambda).
\end{align*}
Here we have used that $W(f_0(\cdot,\lambda
),\overline{f_0(\cdot,\lambda)})=-2i\lambda$ which can be seen by
computing the Wronskian at $\xi =\infty$. In conclusion,
$$\calH_0f_+(\xi,\lambda )=\lambda^2\left[f_0(\xi,\lambda)
+\int_{\xi}^{\infty}G_0(\xi,\eta;\lambda)V_1(\eta)f_+(\eta,\lambda)\, d\eta \right]-V_1(\xi )f_+(\xi,\lambda )$$
or equivalently,
$$\calH f_+(\cdot,\lambda )=\lambda^2f_+(\cdot ,\lambda ).$$
Finally, observe that for $\xi >\lambda^{-1}$ fixed,
$$\sup_{\eta >\xi}|G_0(\xi,\eta ;\lambda )|\les\lambda^{-1}.$$
By the Volterra iteration discussed above, this implies that
$|f_+(\xi,\lambda )-f_0(\xi,\lambda )|\les \lambda^{-1}\xi^{-2}.$ In
particular,
$$f_+(\xi,\lambda )\sim e^{i\lambda\xi} \qquad\mathrm{as}\;
\xi\to\infty$$ For the estimate~\eqref{eq:G0_bd0}, recall the
asymptotic bounds \beeq\label{26}
H^{(+)}_0(x)=1+O_{\R}(x^2)+\frac{2}{\pi}i\log x+i\varkappa
+iO_{\R}(x^2\log x) \eneq as $x\to 0$ where $\varkappa$ is some real
constant, see \cite{AS}. Moreover, $|H^{(+)}_0(x)|\les x^{-\frac12}$
for all $x>1$. Hence,
\begin{align*}
  |G_0(\xi,\eta;\lambda)| &\les (\xi\eta)^{\frac12} |H_0^{(+)}(\lambda
\xi)| |H_0^{(+)}(\lambda \eta)| \\
&\les (\xi\eta)^{\frac12} |\log(\lambda\xi)| \Big(
|\log(\lambda\eta)| \chi_{[\eta\lambda<1]} +
(\eta\lambda)^{-\frac12} \chi_{[\eta\lambda\ge 1]}\Big)
\end{align*}
which implies~\eqref{eq:G0_bd0}.
\end{proof}

\noindent Estimating the oscillatory integrals will require
understanding
$\partial^{k}_{\lambda}\partial^{\ell}_{\xi}f_{\pm}(\xi,\lambda)$,
for $0\leq k+\ell\leq 2$, $W(\lambda)$, $W'(\lambda )$ and thus
$a_{\pm}(\lambda)$, $b_{\pm}(\lambda)$, $a_{\pm}'(\lambda )$ and
$b_{\pm}'(\lambda )$. To obtain asymptotic expansions for all these
functions, we need to know the asymptotic behavior of $u_{j}(\xi)$,
and thereafter that of
$\partial^{k}_{\lambda}\partial^{\ell}_{\xi}u_{j}(\xi,\lambda)$, for
$j=1,\, 2$ and $0\leq k+\ell\leq 2$.

To study the asymptotic behavior of the $u_{j}(\xi ,\lambda)$,  we
use~\eqref{18}. Setting $h_j(\xi,\lambda):=\frac{u_j(\xi ,\lambda
)}{u_j(\xi )}$, for $\xi >0$ we obtain the integral equations
 \begin{align} \label{31'}
h_{0}(\xi ,\lambda
)&=1+\frac{\lambda^2}{u_{0}(\xi)}\int_0^{\xi}[u_1(\xi)u^{2}_0(\eta)-u_0(\xi)u_1(\eta)u_0
(\eta)]h_0(\eta ,\lambda )\, d\eta, \\
\label{31''} h_{1}(\xi ,\lambda
)&=1+\frac{\lambda^2}{u_{1}(\xi)}\int_0^{\xi}[u_1(\xi
)u_0(\eta)u_1(\eta )-u_0(\xi )u_1^2(\eta )] h_{1}(\eta ,\lambda)\,
d\eta
\end{align}
from~\eqref{18}. The first iterates of~\eqref{31'} and~\eqref{31''}
are controlled by the following lemma. The $O(\cdot)$ terms
appearing here will be differentiated later, for now we only control
their size.

\begin{corollary}\label{cor3}
As $\xi\to\infty$,
\begin{align}\label{31}
u_1(\xi )\int_0^{\xi}u_0^2(\eta )\, d\eta -u_0(\xi
)\int_0^{\xi}u_1u_0(\eta )\, d\eta &=
\frac{1}{4}2^{-1/4}\xi^{5/2}+O(\xi^{3/2}\log\xi ) \\
\label{32} u_1(\xi )\int_0^{\xi}u_0u_1(\eta )\, d\eta -u_0(\xi
)\int_0^{\xi}u_1^2(\eta )\, d\eta &=
\frac{1}{4}2^{1/4}\xi^{5/2}\log\xi
\\&\qquad\qquad +c_3\xi^{5/2}+O(\xi^{\frac{3}{2}}\log\xi ) \nonumber
\end{align}
where $c_3\in\R$ is some constant.
\end{corollary}

\begin{proof}
By the asymptotic expressions for $u_0$ and $u_1$,
\begin{align*}
\int_0^{\xi}u_0^2(\eta )\, d\eta & =
2^{-1/2}\int_0^{\xi}\eta \left(1-\frac{c_{\infty}}{\langle\eta\rangle}+O(\langle\eta\rangle^{-2})\right)\,d\eta\\
&=2^{-1/2}\left(\frac{1}{2}\xi^2-c_{\infty}\xi + O(\log\xi )\right)\\
\int_0^{\xi}u_0(\eta )u_1(\eta )\, d\eta & = \int_0^{\xi}\eta
\left(1-\frac{c_{\infty}}{\langle\eta\rangle}+O(\langle\eta\rangle^{-2})\right)
\left(\log\eta +c_2+O(\langle\eta\rangle^{-1})\right)\,d\eta\\
&=\frac{1}{2}\xi^2\log\xi
+\frac{1}{2}\left(c_2-\frac{1}{2}\right)\xi^2+O(\xi\log\xi ).
\end{align*}
Thus,
\begin{align*}
(\ref{31})&=2^{-1/4}\xi^{1/2}(\log\xi +c_2+O(\xi^{-1}\log\xi ))\left(\frac{1}{2}\xi^2+O(\xi )\right)\\
&\quad -2^{-1/4}\xi^{1/2}(1+O(\xi^{-1}))
\left(\frac{1}{2}\xi^2\log\xi+\frac{1}{2}\left(c_2-\frac{1}{2}\right)\xi^2+O(\xi\log\xi )\right)\\
&=2^{-1/4}\xi^{1/2}\left[\frac{1}{4}\xi^2+O(\xi\log\xi )\right]
\end{align*}
Next, compute
\begin{align*}
\int_0^{\xi}u_1^2(\eta )\, d\eta &=\sqrt{2}\,\int_0^{\xi}\eta
(\log^2\eta +2c_2\log\eta +O(\langle\eta\rangle^{-1}\log\eta ))(1+O(\langle\eta\rangle^{-1}))\, d\eta\\
&=\sqrt{2}\,\left(\frac{1}{2}\xi^2\log^2\xi +(2c_2-1)\int_0^{\xi}\eta\log\eta\, d\eta +O(\xi\log^2\xi )\right)\\
&=\sqrt{2}\,\left(\frac{1}{2}\xi^2\log^2\xi
+\frac{2c_2-1}{2}\xi^2\log\xi -\frac{2c_2-1}{4}\xi^2+O(\xi\log^2\xi
)\right)
\end{align*}
Thus, $(\ref{32})$ equals
\begin{align*}
2^{1/4}&\xi^{1/2}(\log\xi +c_2+O(\xi^{-1}))(1+O(\xi^{-1}))
\left(\frac{1}{2}\xi^2\log\xi
+\frac{1}{2}\left(c_2-\frac{1}{2}\right)\xi^2
+O(\xi\log\xi )\right)\\
&-2^{1/4}\xi^{1/2}(1+O(\xi^{-1})) \left(\frac{1}{2}\xi^2\log^2\xi
+\frac{2c_2-1}{2}\xi^2\log\xi -\frac{2c_2-1}{4}\xi^2+O(\xi\log^2\xi
)\right) \\
=2^{1/4}&\xi^{1/2}\bigg\{\frac{1}{2}\xi^2\log^2\xi
+\frac{2c_2-\frac{1}{2}}{2}\xi^2\log\xi +O(\xi\log^2\xi )+\frac{c_2}{2}\left( c_2-\frac{1}{2}\right)\xi^2\\
&-\frac{1}{2}\xi^2\log^2\xi -\frac{2c_2-1}{2}\xi^2\log\xi
+\frac{2c_2-1}{4}\xi^2\bigg\}\end{align*} which finally reduces to
\begin{align*}
2^{1/4}&\sqrt{\xi}\left(\frac{1}{4}\xi^2\log\xi+2^{-1/4}c_3\xi^2+O(\xi\log\xi
)\right)
\end{align*}
as claimed.
\end{proof}

Thus a Volterra iteration and the preceding yields the following
result for the $u_{j}(\xi,\lambda)$'s. The importance of
Corollary~\ref{cor4} lies with the fact that we do not lose
$\log\xi$ factors in the $O(\cdot)$-terms
 as such factors would destroy the dispersive estimate. It is easy to see that carrying
out the Volterra iteration crudely, by putting absolute values
inside the integrals, leads to such $\log\xi$ losses. Therefore, we
actually need to compute the Volterra iterates in~\eqref{eq:volt_it}
explicitly (for the version~($**$)) .

\begin{corollary}\label{cor4}
In the range $1\ll \xi \les \lambda^{-1}$, $j=0,1$,
\begin{align}\label{33}
u_j(\xi ,\lambda )&=u_j(\xi )(1+O((\xi\lambda )^2))\\
\partial_{\xi}u_j(\xi ,\lambda )&=u_j'(\xi )(1+O((\xi\lambda )^2))\nonumber
\\\label{34}
\partial_{\lambda}u_0(\xi ,\lambda )&=\frac{1}{2}2^{-1/4}\lambda (\xi^{5/2}+O(\xi^{3/2}\log\xi ))(1+O((\xi\lambda )^2))\\
\partial_{\lambda}u_1(\xi ,\lambda )&=\frac{1}{2}2^{1/4}\lambda
(\xi^{5/2}\log\xi +c_3\xi^{5/2}+O(\xi^{3/2}\log\xi ))(1+O((\xi\lambda )^2))\nonumber
\\\label{35}
\partial^2_{\lambda\xi}u_0(\xi ,\lambda )&=\frac{5}{4}2^{-1/4}\lambda (\xi^{3/2}+O(\xi^{1/2}\log\xi ))(1+O((\xi\lambda )^2))\\
\partial^2_{\lambda\xi}u_1(\xi ,\lambda )&=\frac{5}{4}2^{1/4}\lambda (\xi^{3/2}\log\xi +\frac{2}{5}\xi^{3/2}\nonumber \\
&\qquad\qquad\qquad +c_3\xi^{3/2}+O(\xi^{1/2}\log\xi
))(1+O((\xi\lambda )^2))\nonumber
\end{align}
If $|\xi|\les 1$, then $  |u_j(\xi,\lambda)|\les 1$,
$|\partial_\lambda u_j(\xi,\lambda)|+ |\partial^2_{\lambda\xi}
u_j(\xi,\lambda)|\les \lambda$ for $j=0,1$.
\end{corollary}

\begin{proof}
We sketch the proof of this somewhat computational lemma, for the
function $u_{1}(\xi,\lambda)$ since the argument for
$u_{0}(\xi,\lambda)$ is completely analogous and in fact easier. The
proof of the first equality in \eqref{33} is based on the Volterra
integral equation \eqref{31''} \beeq h_{1}(\xi ,\lambda
)=1+\lambda^2\int_0^{\xi}[\frac{u_1(\xi )u_0(\eta)u_1(\eta )-u_0(\xi
)u_1^2(\eta )}{u_{1}(\xi)}] h_{1}(\eta ,\lambda)\, d\eta \eneq and
its derivatives in both $\xi$ and $\lambda$ and the Volterra
iteration, for which we also need to use Corollary~\ref{cor3}. The
iteration will produce a solution which is given by
\begin{align*}
&h_{1}(\xi ,\lambda) =  1 + \sum_{n=1}^\infty \lambda^{2n}
\int_0^\xi \int_{0}^{\xi_{1}}\ldots \int_{0}^{\xi_{n-1}}
\frac{u_1(\xi
)u_0(\xi_{1})u_1(\xi_{1} )-u_0(\xi )u_1^2(\xi_{1} )}{u_{1}(\xi)}\cdots\\
&\frac{u_1(\xi_{n-1} )u_0(\xi_{n})u_1(\xi_{n} )-u_0(\xi_{n-1}
)u_1^2(\xi_{n})}{u_{1}(\xi_{n-1} )} \, d\,\xi_{n}\ldots d\,\xi_1 = \\
& 1+ \lambda^{2}\int_0^\xi  \frac{u_1(\xi )u_0(\xi_{1})u_1(\xi_{1}
)-u_0(\xi )u_1^2(\xi_{1} )}{u_{1}(\xi)} d\,\xi_{1}+\\
&\lambda^{4}\int_{0}^{\xi}\int_{0}^{\xi_{1}} \frac{u_1(\xi
)u_0(\xi_{1})u_1(\xi_{1} )-u_0(\xi )u_1^2(\xi_{1}
)}{u_{1}(\xi)}\;\frac{u_1(\xi_{1} ) u_0(\xi_{2})u_1(\xi_{2}
)-u_0(\xi_{1} )u_1^2(\xi_{2})}{u_{1}(\xi_{1} )} \, d\,\xi_{2}\,
d\,\xi_1 +\cdots
\end{align*}
Therefore, \eqref{32} and the equalities
\begin{align*}
u_0(\xi )&=2^{-1/4}\xi^{1/2}\left(1-\frac{c_{\infty}}{2\xi}+O(\xi^{-2})\right)\nonumber\\
u_1(\xi
)&=2^{1/4}\xi^{1/2}\left(1-\frac{c_{\infty}}{2\xi}+O(\xi^{-2})\right)\left(\log\xi+c_2+O(\xi^{-1})\right)
\end{align*}
yield
\begin{align*}
h_{1}(\xi ,\lambda)= & 1+
\frac{\lambda^{2}}{u_{1}(\xi)}\Big(\frac{1}{4}2^{1/4}\xi^{5/2}\log\xi
+c_3\xi^{5/2}+O(\xi^{\frac{3}{2}}\log\xi )\Big)\\
&+\lambda^{4}\Big\{\int_{0}^{\xi}
u_0(\xi_{1})\Big[\frac{1}{4}2^{1/4}\xi_{1}^{5/2}\log\xi_{1}
+c_3\xi_{1}^{5/2}+O(\xi_{1}^{\frac{3}{2}}\log\xi_{1})\Big]\, d\xi_{1}-\\
&\frac{u_{0}(\xi)}{u_{1}(\xi)}\int_{0}^{\xi}u_1(\xi_{1}
)\Big[\frac{1}{4}2^{1/4}\xi_{1}^{5/2}\log\xi_{1}
+c_3\xi_{1}^{5/2}+O(\xi_{1}^{\frac{3}{2}}\log\xi_{1})\Big]\,d\xi_{1}\Big\}
+\cdots \\
&= 1+ O(\lambda^{2}\xi^{2}),
\end{align*}
since we are assuming that $1\ll\xi \les\lambda^{-1}$. The point to
notice here is that terms involving $\xi^4\log\xi$ (the leading
orders) after the integration cancel. Furthermore, we obtain the
usual $n!$ gain from the Volterra iteration, see Lemma~\ref{4'},
from repeated integration of powers rather than from symmetry
considerations. Hence $u_{1}(\xi,\lambda)=
u_{1}(\xi)(1+O(\lambda^{2}\xi^{2}))$ in that range. To deal with the
derivatives, it is more convenient to  directly  differentiate the
integral equation \eqref{18} for $u_{1}(\xi,\lambda)$ with respect
to $\xi$ and/or $\lambda$, which yields, respectively,
\begin{align}\label{iter1}
\partial_{\xi}u_{1}(\xi,\lambda) &= \partial_{\xi} u_{1}(\xi) + \lambda ^2 \int_{0}^{\xi}
[\partial_{\xi}u_1 (\xi) u_{0}(\eta)-u_{1}(\eta)\partial_{\xi}
u_{0}(\xi)]u_{1}( \eta, \lambda) d\eta \\
\label{iter2}
\partial_{\lambda}u_{1}(\xi,\lambda)&=  2\lambda \int_{0}^{\xi}[u_1
(\xi) u_{0}(\eta)-u_{1}(\eta) u_{0}(\xi)]u_{1}( \eta, \lambda) d\eta
\\ & \qquad +\lambda^{2}\int_{0}^{\xi}[u_1 (\xi)
u_{0}(\eta)-u_{1}(\eta)u_{0}(\xi)]\partial_{\lambda}u_{1}(\eta,
\lambda) d\eta, \nonumber
\end{align}
as well as
\begin{align}
 \label{iter3}
\partial^{2}_{\lambda\, \xi}u_{1}(\xi,\lambda)&=  2\lambda
\int_{0}^{\xi}[\partial_{\xi}u_1 (\xi)
u_{0}(\eta)-u_{1}(\eta)\partial_{\xi} u_{0}(\xi)]u_{1}( \eta,
\lambda) d\eta \nonumber\\ &\qquad
+\lambda^{2}\int_{0}^{\xi}[\partial_{\xi}u_1 (\xi)
u_{0}(\eta)-u_{1}(\eta)\partial_{\xi}u_{0}(\xi)]\partial_{\lambda}u_{1}(\eta,
\lambda) d\eta.
\end{align}
In dealing with \eqref{iter1}, we simply plug in the information
from the first equality of \eqref{33} and calculate the resulting
integral. For \eqref{iter2}, we observe that by \eqref{32} the term
\[ 2\lambda \int_{0}^{\xi}[u_1 (\xi) u_{0}(\eta)-u_{1}(\eta)
u_{0}(\xi)]u_{1}( \eta, \lambda) d\eta\] is equal to $\lambda
(\frac12 2^{1/4}\xi^{5/2}\log\xi +2c_3\xi^{5/2}+O(\xi^{3/2}\log\xi
)) $. Therefore to solve \eqref{iter2}, one needs to run the
Volterra iteration with this expression as the first iterate. The
treatment of \eqref{iter3} is similar to that of \eqref{iter2} and
we skip the details. The case of $|\xi|\les 1$ is left to the
reader.
\end{proof}

We now turn to $f_{\pm}(\xi,\lambda )$ as well as $a_{\pm}$,
$b_{\pm}(\lambda )$.

\begin{lemma}\label{lemma8}
If $\lambda >0$ is small, and $|\log\lambda|^2\le \xi \ll
\lambda^{-1}$, then \beeq\label{35'} f_+(\xi,\lambda
)=f_0(\xi,\lambda )+O(\xi^{-1/2}\lambda^{\frac{1}{2}-\eps})\nonumber
\eneq with $\eps
>0$ arbitrary.
\end{lemma}

\begin{proof}
Let
\[
m(x):= \sqrt{x}\, |\log x|\, \chi_{[0<x<1]} +
\chi_{[x>1]}
\]
Then, in view of the asymptotic behavior of $H_0^{(+)}$,
\[
|f_0(\xi,\lambda)|\les m(\xi\lambda)
\] and thus also
\[
|G_0(\xi,\eta;\lambda)|\les \lambda^{-1}\,
m(\xi\lambda)m(\eta\lambda)
\]
 We claim that also
\begin{equation}
  \label{eq:f+claim} |f_{+}(\xi,\lambda)|\les m(\xi\lambda)
\end{equation}
With $g(x;\lambda):= f_{+}(\xi,\lambda)/m(\xi\lambda)$, we obtain
the integral inequality
\[
g(\xi,\lambda) \le C + C\int_\xi^\infty \lambda^{-1}|V_1(\eta)|
m(\eta\lambda)^2 g(\eta,\lambda)\, d\eta
\]
for some absolute constant $C$. Since by our assumption on~$\xi$,
\[
\int_\xi^\infty \lambda^{-1}|V_1(\eta)| m(\eta\lambda)^2 \, d\eta
\les \int_\xi^\infty \lambda^{-1} \eta^{-3} m(\eta\lambda)^2 \,
d\eta\les \xi^{-1} |\log\lambda|^2 + \lambda \les 1,
\]
the claim follows from Lemma~\ref{4'}.  We observed above that,
see~\eqref{eq:G0_bd0}, \beeq |G_0(\xi,\eta;\lambda
)|\les\sqrt{\xi\eta}\,|\log\lambda |^2\,\chi_{[\xi <\eta
<\lambda^{-1}]}+\sqrt{\frac{\xi}{\lambda}}\,|\log\lambda
|\,\chi_{[\eta >\lambda^{-1}]}\nonumber \eneq Thus integrating and
taking $1\ll \xi \ll \lambda^{-1}$ into account, we obtain
from~\eqref{eq:f+claim} that
\begin{align*}
\bigg|\int_{\xi}^{\infty}G_0(\xi,\eta;\lambda )V_1(\eta )f_+(\eta,\lambda )\, d\eta \bigg|
&\les\int_{\xi}^{\lambda^{-1}}\sqrt{\xi\eta}\,|\log\lambda |^2\,\eta^{-3}\sqrt{\eta\lambda}\,|\log\lambda |\, d\eta\\
&+\int_{\lambda^{-1}}^{\infty}\sqrt{\frac{\xi}{\lambda}}\,|\log\lambda
|\,\eta^{-3}\, d\eta \les\,\xi^{-1/2}\lambda^{\frac{1}{2}-\eps},
\end{align*}
as claimed.
\end{proof}

We can now state our asymptotic expansion of $a_+$ and $b_+$. In
what follows, $O(\cdot)$ terms are complex-valued unless stated to
the contrary (which will be denoted by $O_\R(\cdot)$).

\begin{lemma}\label{lemma9}
With $\eps>0$ arbitrary, small, and fixed,
\begin{equation}\begin{aligned}\label{36}
a_+(\lambda )&=2^{1/4}c_0\sqrt{\lambda}(1+ic_1\log\lambda +ic_3)+O(\lambda^{1-\eps})\\
b_+(\lambda
)&=i2^{-1/4}c_0c_1\sqrt{\lambda}+O(\lambda^{1-\eps}),
\end{aligned}
\end{equation}
as $\lambda\to0+$,  where
$c_0=\sqrt{\frac{\pi}{2}}e^{i\frac{\pi}{4}}$, $c_1=\frac{2}{\pi}$,
and $c_3$ is some real constant.
\end{lemma}

\begin{proof}
By Corollary~\ref{cor1} we have $a_+(\lambda )=f_+(\xi,\lambda
)u'_1(\xi,\lambda )-f'_+(\xi,\lambda )u_1(\xi,\lambda )$. Hence
Lemma~\ref{lemma8} and Corollary~\ref{cor4} applied to
$\xi=\lambda^{-1/2}$ yield,
\begin{align*}
c_0^{-1}2^{1/4}a_+&= \sqrt{\lambda\xi}\,H_0(\xi\lambda )\frac{1}{2}\xi^{-1/2}(\log\xi +c_2+2)\\
&-\left(\frac{1}{2}\xi^{-1/2}\sqrt{\lambda}\,H_0(\xi\lambda
)+\sqrt{\xi\lambda}\,H_0'(\xi\lambda )
\lambda \right)\xi^{1/2}(\log\xi +c_2)+O(\lambda^{1-\eps})\\
&= \sqrt{\lambda}\,H_0(\xi\lambda )-\sqrt{\xi\lambda}\: \frac{ic_1}{\xi}\sqrt{\xi}(\log\xi +c_2)+O(\lambda^{1-\eps})\\
&= \sqrt{\lambda}(1+ic_1\log (\xi\lambda )+i\varkappa -ic_1\log\xi -ic_1c_2)+O(\lambda^{1-\eps})\\
&= \sqrt{\lambda}(1+ic_1\log\lambda +ic_3)+O(\lambda^{1-\eps}),
\end{align*}
as claimed.  Note that $c_3=\varkappa -c_1c_2$. Similarly,
\begin{align*}
-c_0^{-1}2^{\frac{1}{4}}b_+&=\sqrt{\lambda\xi}\, H_0(\xi\lambda )
\frac{1}{2}\xi^{-1/2}-\xi^{1/2}\left(\frac{1}{2}\xi^{-1/2}
\sqrt{\lambda}\, H_0(\lambda\xi )+\sqrt{\xi\lambda}\, H_0'(\xi\lambda )\lambda \right)\\
&\qquad\qquad\qquad +O(\lambda^{1-\eps})\\
&=-\xi\sqrt{\lambda}\:\frac{ic_1}{\xi\lambda}\,\lambda+O(\lambda^{1-\eps})
=-ic_1\sqrt{\lambda}+O(\lambda^{1-\eps}),
\end{align*}
and the lemma follows.
\end{proof}

Using the expressions for $a_{+}$ and $b_{+}$ above, we obtain the
following

\begin{corollary}\label{cor5}
Let $\lambda >0$ be small.  Then \beeq\label{37} f_+(\xi, \lambda
)=c_0\sqrt{\lambda\langle\xi\rangle}\left(1+ic_1\log
(\lambda\langle\xi\rangle
)+ic_4+O(\lambda^{\frac{1}{2}-\eps})+O(\langle\xi\rangle^{-1}\log\langle\xi\rangle
)\right) \eneq for $0<\xi <\lambda^{-1}$, whereas for
$-\lambda^{-1}<\xi <0,$ \beeq\label{38} f_+(\xi, \lambda
)=c_0\sqrt{\lambda\langle\xi\rangle}\left(1+ic_1\log
(\lambda\langle\xi\rangle^{-1}
)+ic_5+O(\lambda^{\frac{1}{2}-\eps})+O(\langle\xi\rangle^{-1}\log\langle\xi\rangle
)\right) \eneq
Here $c_1$ is as above and $c_4, c_5$ are real constants.
\end{corollary}

\begin{proof}
This follows by inserting our asymptotic expansions for $a_+(\lambda
)$, $b_+(\lambda )$,
 $u_0(\xi,\lambda )$, and $u_1(\xi,\lambda )$ into (\ref{20}).
\end{proof}

 We also need some information about certain partial derivatives of
 $f_{+}(\xi,\lambda)$. This is provided by

\begin{lemma}\label{lemma10}
For $\lambda >0$ small and $|\log\lambda|^2\le \xi \ll \lambda^{-1}$
we have
\begin{align*}
\partial_{\xi}f_+(\xi,\lambda )&=\partial_{\xi} f_0(\xi,\lambda )+O(\xi^{-3/2}\lambda^{\frac{1}{2}-\eps})\\
\partial_{\lambda}f_+(\xi,\lambda )&=\partial_{\lambda} f_0(\xi,\lambda )+O(\xi^{-1/2}\lambda^{-\frac{1}{2}-\eps})\\
\partial^2_{\xi\lambda}f_+(\xi,\lambda )&=\partial^2_{\xi\lambda} f_0(\xi,\lambda )+O(\xi^{-3/2}\lambda^{-\frac{1}{2}-\eps})
\end{align*}
with $\eps >0$ arbitrary.
\end{lemma}

\begin{proof}
This follows by taking derivatives in Lemma~\ref{lemma8}.
\end{proof}

To be able to carry out the analysis, one also needs to understand
the derivative of the Wronskian. To that end we have

\begin{corollary}\label{cor6}
Then,  with $\eps>0$ arbitrary but fixed,
\begin{equation}
  \begin{aligned}\label{39}
a_+'(\lambda )&=\frac{1}{2}2^{1/4}c_0\lambda^{-1/2}(1+ic_3+2ic_1+ic_1\log\lambda )+O(\lambda^{-\eps})\\
b'_+(\lambda
)&=\frac{i}{2}2^{-1/4}c_0c_1\lambda^{-1/2}+O(\lambda^{-\eps})
\end{aligned}
\end{equation}
as $\lambda\to0+$.
\end{corollary}

\begin{proof}
In view of the preceding,
\begin{align}\label{40}
a'_+(\lambda )&= W(\partial_{\lambda}f_+,u_1)+W(f_+,\partial_{\lambda}u_1)\nonumber\\
&= W(\partial_{\lambda}f_0,u_1)+W(f_0,\partial_{\lambda}u_1)+O(\lambda^{-\eps})\nonumber\\
&= \partial_{\lambda}[c_0\sqrt{\lambda\xi}H_0(\lambda\xi )]\left(\frac{1}{2}\xi^{-1/2}(\log\xi +c_2)+\xi^{-1/2}\right)2^{1/4}\\
&\quad-\partial^2_{\lambda\xi}[c_0\sqrt{\lambda\xi}H_0(\lambda\xi )]\xi^{1/2}(\log\xi +c_2)\cdot 2^{1/4}\nonumber\\
&\quad+c_0\sqrt{\lambda\xi}H_0(\lambda\xi )\cdot\frac{5}{4}\cdot 2^{1/4}\lambda \left(\xi^{3/2}\log\xi +\left(\frac{2}{5}+c_3\right)\xi^{3/2}\right)\nonumber\\
&\quad-c_0\partial_{\xi}[\sqrt{\lambda\xi}H_0(\lambda\xi
)]\frac{1}{2}2^{1/4}\lambda (\xi^{5/2}\log\xi
+c_3\xi^{5/2})+O(\lambda^{-\eps}).\nonumber
\end{align}
Evaluating at $\xi=\lambda^{-1/2}$, one obtains that the third and
fourth terms in (\ref{40}) are $O(\lambda^{\frac{1}{2}-\eps})$, and
thus error terms.  Thus,
\begin{align*}
2^{-1/4}c_0^{-1}a_+'(\lambda )&=\left(\frac{1}{2}\lambda^{-1/2}(1+ic_1\log (\lambda\xi )+i\varkappa )+ic_1\lambda^{-1/2}\right)\left(\frac{1}{2}(c_2+\log\xi )+1\right)\\
&-\left(\frac{1}{4}\lambda^{-1/2}(1+ic_1\log (\lambda\xi
)+i\varkappa )+ic_1\lambda^{-1/2}\right)(\log\xi +c_2)
+O(\lambda^{-\eps})\\
&=\frac{1}{2}\lambda^{-1/2}(1+ic_1\log (\lambda\xi )+i\varkappa )+ic_1\lambda^{-1/2}
\end{align*}
which further simplifies to
\begin{align*}
&\quad-\frac{ic_1}{2}\lambda^{-1/2}(\log\xi +c_2)+O(\lambda^{-\eps})\\
&=\frac{1}{2}\lambda^{-1/2}(1+ic_1\log\lambda+i\varkappa+2ic_1-ic_1c_2)+O(\lambda^{-\eps})\\
&=\frac{1}{2}\lambda^{-1/2}(1+ic_3+2ic_1+ic_1\log\lambda
)+O(\lambda^{-\eps}).
\end{align*}
Similarly,
\begin{align*}
2^{-1/4}c_0^{-1}b_+'(\lambda )&=
\frac{1}{2}\left(\frac{1}{2}\lambda^{-1/2}(1+ic_1\log (\lambda\xi
)+i\varkappa )
+ic_1\lambda^{-1/2}\right)\\
&\quad-\left(\frac{1}{4}\lambda^{-1/2}(1+ic_1\log (\lambda\xi )+i\varkappa) +ic_1\lambda^{-1/2}\right)+O(\lambda^{-\eps})\\
&= -\frac{1}{2}ic_1\lambda^{-1/2}+O(\lambda^{-\eps}),
\end{align*}
as claimed.
\end{proof}

\begin{remark}
  \label{rem:left-right} Recall that this analysis was carried out
  assuming that $\calM$ is conical on the right. If $\calM$ is conical on the left, then
  the same analysis applies. In fact, if $\calM$ is symmetric, i.e.,
  $r(x)=r(-x)$,  then by Corollary~\ref{cor1}
  $a_-(\lambda)=a_+(\lambda)$ and $b_-(\lambda)=-b_+(\lambda)$. If it
  is not symmetric but still conical at both ends, then these
  relations still hold for the asymptotic expansions. i.e.,
  \begin{align*}
a_-(\lambda )&=2^{1/4}c_0\sqrt{\lambda}(1+ic_1\log\lambda +ic_3)+O(\lambda^{1-\eps})\\
b_-(\lambda
)&=-i2^{-1/4}c_0c_1\sqrt{\lambda}+O(\lambda^{1-\eps}),\nonumber
\end{align*}
as $\lambda\to0+$ where $c_0$ etc.~are as in Lemma~\ref{lemma9}. The
same of course applies to $a_-'$ and $b_-'$.
\end{remark}

We end the perturbative analysis with a description of the
oscillatory behavior of $f_+(\xi,\lambda)$ for $\lambda\xi>1$.

\begin{lemma}\label{lemma11}
Let $m_+(\xi,\lambda ):=e^{-i\lambda\xi}f_+(\xi,\lambda )$.  Then,
provided $\lambda >0$ is small and $\lambda\xi >1$,
\begin{align}\label{41}
|m_+(\xi,\lambda )-1|&\les (\lambda\xi )^{-1}\\
|\partial_{\lambda}m_+(\xi,\lambda )|&\les
\lambda^{-2}\xi^{-1}\nonumber
\end{align}
\end{lemma}

\begin{proof}
From (\ref{24}), and with $m_0(\xi,\lambda
)=e^{-i\lambda\xi}f_0(\xi,\lambda )$, \beeq\label{42}
m_+(\xi,\lambda )=m_0(\xi,\lambda
)+\int_{\xi}^{\infty}\widetilde{G}_0(\xi,\eta;\lambda )V_1(\eta
)m_+(\eta,\lambda )\, d\eta \eneq where \beeq\label{43}
\widetilde{G}_0(\xi,\eta;\lambda )=\frac{m_0(\xi,\lambda
)\overline{m_0(\eta,\lambda )}-e^{-2i(\xi
-\eta)\lambda}\overline{m_0(\xi,\lambda )}m_0(\eta,\lambda
)}{-2i\lambda} \eneq Now, by asymptotic properties of the Hankel
functions,
$$m_0(\xi,\lambda )=1+O((\xi\lambda )^{-1})$$
where the $O$-term behaves like a symbol.\footnote{In fact,
$m_0(\xi,\lambda )=1+O_\R((\xi\lambda )^{-2})+iO_\R((\xi\lambda
)^{-1})$.} Inserting this bound into (\ref{43}) yields
$$|\widetilde{G}_0(\xi,\eta;\lambda )|\les\eta$$
provided $\eta >\xi >\lambda^{-1}.$ Thus, from (\ref{42}),
$$|m_+(\xi,\lambda )-m_0(\xi,\lambda )|\les\xi^{-1}$$
and thus, for all $\xi\lambda >1$,
$$|m_+(\xi,\lambda )-1|\les (\xi\lambda )^{-1}$$
as claimed.\\
Next, one checks that for $\eta >\xi >\lambda^{-1}$,
$$|\partial_{\lambda}\widetilde{G}_0(\xi,\eta;\lambda )|\les\frac{\eta}{\lambda}.$$
Thus, for all $\lambda\xi >1$,
\begin{align*}
|\partial_{\lambda}m_+(\xi,\lambda )|\les
&\lambda^{-2}\xi^{-1}+\int_{\xi}^{\infty}|\partial_{\lambda}\widetilde{G}_0(\xi,\eta;\lambda
)|\eta^{-3}\, d\eta
 + \int_{\xi}^{\infty}\eta^{-2}|\partial_{\lambda}m_+(\eta\lambda )|\, d\eta\\
\les &
\lambda^{-2}\xi^{-1}+\lambda^{-1}\xi^{-1}+\int_{\xi}^{\infty}\eta^{-2}|\partial_{\lambda}m_+(\eta,\lambda
)|\, d\eta \les  \lambda^{-1}(\lambda\xi )^{-1},
\end{align*}
as claimed.
\end{proof}

\subsection{The Wronksian $W(\lambda)$ for conical
ends, $d=1, n=0$}

In view of our asymptotic analysis of $a_{\pm}$ and $b_\pm$ and an
explicit expression for the Wronskian $W(\lambda)$ in terms of these
functions, see Corollary~\ref{cor1}, we now derive the following
important fact.

\begin{corollary}\label{cor7}
As $\lambda\to0+$,
\begin{align*}
W(\lambda )&=2\lambda \left( 1+ic_3+i\frac{2}{\pi}\log\lambda \right)+O(\lambda^{\frac{3}{2}-\eps})\\
W'(\lambda )&=2\left(
1+ic_3+i\frac{2}{\pi}+i\frac{2}{\pi}\log\lambda \right)
+O(\lambda^{\frac{1}{2}-\eps})
\end{align*}
with $\eps >0$ arbitrary.
\end{corollary}

\begin{proof}
 Follows immediately from
$$W(\lambda )=(a_-b_+-a_+b_-)(\lambda )$$
and (\ref{36}), (\ref{39}). See Remark~\ref{rem:left-right}.
\end{proof}

\section{The oscillatory integral estimates for $d=1, n=0$}

We now commence with proving the main oscillatory integral
estimate~(\ref{eq:schr_oscill}) and \eqref{eq:wave_oscill} for small
energies. Thus, let $\chi$ be a smooth cut-off function to small
energies, i.e., $\chi (\lambda )=1$ for small $|\lambda |$ and
$\chi$ vanishes outside a small interval around zero.  In addition,
we introduce the smooth cut-off functions $\chi_{[|\xi\lambda |<1]}$
and $\chi_{[|\xi\lambda |>1]}$ which form a partition of unity
adapted to these intervals.

\begin{lemma}\label{lemma12}
For all $t>0$ \begin{align}\label{44} \sup_{\xi,\xi
'}\bigg|\int_0^{\infty}e^{it\lambda^2}\lambda
\frac{\chi(\lambda;\xi,\xi')}{\langle\xi\rangle\langle\xi
'\rangle)^{1/2}}  \mathrm{Im}\left[\frac{f_+(\xi,\lambda )f_-(\xi
',\lambda )}{W(\lambda )}\right]\, d\lambda \bigg|\les \la t\ra^{-1} \\
\label{44'} \sup_{\xi,\xi '}\bigg|\int_0^{\infty}e^{\pm
it\lambda}\lambda\,
\frac{\chi(\lambda;\xi,\xi')}{(\langle\xi\rangle\langle\xi
'\rangle)^{1/2}}\mathrm{Im}\left[\frac{f_+(\xi,\lambda )f_-(\xi
',\lambda )}{W(\lambda )}\right]\, d\lambda \bigg|\les \la t\ra^{-1}
\end{align} where $\chi(\lambda;\xi,\xi'):=\chi (\lambda
)\chi_{[|\xi\lambda |<1,|\xi '\lambda |<1]}$.
\end{lemma}

\begin{proof} We shall first assume for simplicity that $\calM$ is
symmetric, i.e., $r(x)=r(-x)$. The general case will be discussed at the end of the proof.  We first observe the following:
\begin{align*}
\text{Im}&\left[\frac{f_+(\xi,\lambda )f_-(\xi ',\lambda )}{W(\lambda )}\right]\\
&=\text{Im}\left[\frac{(a_+(\lambda )u_0(\xi,\lambda )+b_+(\lambda )u_1(\xi,\lambda ))(a_+(\lambda )u_0(\xi ',\lambda )-b_+(\lambda )u_1(\xi ',\lambda ))}{-2a_+b_+(\lambda )}\right]\\
& = -\frac{1}{2}\text{Im}\left(\frac{a_+}{b_+}(\lambda
)\right)u_0(\xi,\lambda )u_0(\xi ',\lambda
)+\frac{1}{2}\text{Im}\left(\frac{b_+}{a_+}(\lambda
)\right)u_1(\xi,\lambda )u_1(\xi',\lambda).
\end{align*}
Further, by (\ref{36}), with $\eps>0$ arbitrary but fixed,
\begin{align*}
-\frac{1}{2}\text{Im}\left(\frac{a_+}{b_+}(\lambda )\right)&
=\frac{\pi}{2\sqrt{2}\,}\text{Re}\left[\frac{1+ic_1\log\lambda
+ic_3+O(\lambda^{\frac{1}{2}-\eps})}{1+O(\lambda^{\frac{1}{2}-\eps})}\right]\\
&=O_\R(\lambda^{\frac{1}{2}-\eps})+\frac{\pi}{2\sqrt{2}\,}
\end{align*}
and by Corollary~\ref{cor6}, the $O$-term can be formally
differentiated, i.e.,
$$\frac{d}{d\lambda}\left\{-\frac{1}{2}\text{Im}\left(\frac{a_+}{b_+}(\lambda )\right)\right\}=O_\R(\lambda^{-\frac{1}{2}-\eps}).$$
Similarly,
$$\frac{1}{2}\text{Im}\left(\frac{b_+}{a_+}(\lambda )\right)=
-\frac{\sqrt{2}\,}{\pi}\,\frac{1}{1+(c_3+c_1\log\lambda)^2}+O_\R(\lambda^{\frac{1}{2}-\eps})$$
which can again be formally differentiated.\\
By the estimates of Corollary~\ref{cor4}, provided $|\xi\lambda
|+|\xi '\lambda |<1$,
\begin{align*}
|u_0(\xi,\lambda )u_0(\xi ',\lambda
)|&\les\sqrt{\langle\xi\rangle\langle\xi '\rangle}\\
|\partial_{\lambda}[u_0(\xi,\lambda )u_0(\xi ',\lambda )]|&\les \lambda
\left(\langle\xi\rangle^{5/2}\langle\xi '\rangle^{1/2}+\langle\xi '\rangle^{5/2}\langle\xi\rangle^{1/2}\right)\\
&\les \lambda\sqrt{\langle\xi\rangle\langle\xi
'\rangle}(\langle\xi\rangle^2+\langle\xi '\rangle^2).
\end{align*}
Similarly,
\begin{align*}   |u_1(\xi,\lambda )u_1(\xi ',\lambda
)|&\les\sqrt{\langle\xi\rangle\langle\xi '\rangle}\log(2+\la\xi\ra) \log(2+\la\xi'\ra)  \\
|\partial_{\lambda}[u_1(\xi,\lambda )u_1(\xi ',\lambda )]|
&\les\lambda\sqrt{\langle\xi\rangle\langle\xi
'\rangle}(\langle\xi\rangle^2+\langle\xi
'\rangle^2)\log(2+\la\xi\ra) \log(2+\la\xi'\ra)
\end{align*}
Passing absolute values inside \eqref{44} and~\eqref{44'} shows that
these expressions are dominated by
\begin{equation}\label{eq:0block}\begin{aligned} & \int_0^{\infty}\bigg|
\chi (\xi,\xi';\lambda ) (\langle\xi\rangle\langle\xi '\rangle
)^{-1/2}\text{Im}\left(\frac{a_+}{b_+}(\lambda
)\right)u_0(\xi,\lambda )
u_0(\xi ',\lambda )\bigg|\, d\lambda \\
&\quad + \int_0^{\infty}\bigg|\chi (\xi,\xi';\lambda)
(\langle\xi\rangle\langle\xi '\rangle
)^{-1/2}\text{Im}\left(\frac{b_+}{a_+}(\lambda )\right)
u_1(\xi,\lambda )u_1(\xi ',\lambda )\bigg|\, d\lambda
\end{aligned}
\end{equation}
which is bounded by an absolute constant. To obtain decay in $t$, we
integrate by parts.  Integrating by parts in (\ref{44}) yields that
it  is dominated by
\begin{equation}\label{eq:tblock}\begin{aligned} & t^{-1}\int_0^{\infty}\bigg|\partial_{\lambda}
\Big[\chi (\xi,\xi';\lambda ) (\langle\xi\rangle\langle\xi '\rangle
)^{-1/2}\text{Im}\left(\frac{a_+}{b_+}(\lambda
)\right)u_0(\xi,\lambda )
u_0(\xi ',\lambda )\Big]\bigg|\, d\lambda \\
&\quad + t^{-1}\int_0^{\infty}\bigg|\partial_{\lambda} \Big[\chi
(\xi,\xi';\lambda) (\langle\xi\rangle\langle\xi '\rangle
)^{-1/2}\text{Im}\left(\frac{b_+}{a_+}(\lambda )\right)
u_1(\xi,\lambda )u_1(\xi ',\lambda )\Big]\bigg|\, d\lambda
\end{aligned}
\end{equation}
Using the bounds we derived above these expressions can be seen to
be $\les t^{-1}$ and \eqref{44} holds. For~\eqref{44'} we write
$e^{it\lambda}=(it)^{-1}\partial_\lambda e^{it\lambda}$ and
integrate by parts; this yields that the left-hand side
of~\eqref{44'} is dominated by  the exact same terms as
in~\eqref{eq:tblock} (in fact, with an extra $\lambda$).

\noindent If $\calM$ is not symmetric, then the asymptotics of the previous section allow for the
following conclusion (in very much the same way as in the symmetric case):
\begin{align*}
 \Im \left[\frac{f_+(\xi,\lambda )f_-(\xi ',\lambda )}{W(\lambda )}\right]
& = \big(\gamma_0 + O_\R(\lambda^{\frac12-\eps})\big) u_0(\xi,\lambda) u_0(\xi',\lambda) \\
&\quad +
\big(\frac{\gamma_1}{1+(c_3+c_1\log\lambda)^2}
+ O_\R(\lambda^{\frac12-\eps})\big) u_1(\xi,\lambda) u_1(\xi',\lambda)\\
& \quad +O_\R(\lambda^{\frac12-\eps}) (u_0(\xi,\lambda) u_1(\xi',\lambda) + u_1(\xi,\lambda) u_0(\xi',\lambda))
\end{align*}
where $\gamma_0, \gamma_1$ are nonzero real constants (in fact, the same as in the symmetric case). With
this representation in hand, the oscillatory integrals are estimated exactly as in the symmetric case.
\end{proof}

Next, we consider the case $|\xi\lambda |>1$ and $|\xi '\lambda
|>1$.  With the convention that $f_{\pm}(\xi,-\lambda
)=\overline{f_{\pm}(\xi,\lambda )}$ we can remove the imaginary part
in (\ref{eq:schr_oscill}) and integrate $\lambda$ over the whole
axis.  We shall follow this convention hence forth. To estimate the
oscillatory integrals, we shall repeatedly use the following version
of stationary
phase, see Lemma~2 in \cite{Sch} for the proof.

\begin{lemma}\label{lemma12'}
Let $\phi (0)=\phi '(0)=0$ and $1\leq\phi ''\leq C$.  Then
\beeq\label{46} \bigg|\int_{-\infty}^{\infty}e^{it\phi (x)}a(x)\,
dx\bigg|\les\delta^2\left\{\int\frac{|a(x)|}{\delta^2+|x|^2}\,
dx+\int_{|x|>\delta}\frac{|a'(x)|}{|x|}\, dx\right\} \eneq
where $\delta =t^{-1/2}$.
\end{lemma}

\noindent Using Lemma~\ref{lemma12'} we can prove the following:
\begin{lemma}\label{lemma13} With $\chi(\lambda;\xi,\xi')=\chi (\lambda
)\chi_{[|\xi\lambda |>1,|\xi '\lambda |>1]}$,
\begin{align}\label{45} \sup_{\xi >0>\xi
'}\bigg|\int_{-\infty}^{\infty}e^{it\lambda^2}\lambda\chi(\lambda;\xi,\xi')(\langle\xi\rangle\langle\xi
'\rangle)^{-1/2}\frac{f_+(\xi,\lambda )f_-(\xi ',\lambda
)}{W(\lambda )}\, d\lambda \bigg|\les t^{-1} \\
\label{45'} \sup_{\xi >0>\xi '}\bigg|\int_{-\infty}^{\infty}e^{\pm
it\lambda}\lambda\chi(\lambda;\xi,\xi')(\langle\xi\rangle\langle\xi
'\rangle)^{-1/2}\frac{f_+(\xi,\lambda )f_-(\xi ',\lambda
)}{W(\lambda )}\, d\lambda \bigg|\les t^{-\frac12}
\end{align}
for all $t>0$.
\end{lemma}

\begin{proof}
Writing $f_+(\xi,\lambda )=e^{i\xi\lambda}m_+(\xi,\lambda )$,
$f_-(\xi,\lambda )=e^{-i\xi\lambda}m_-(\xi,\lambda )$ as in
Lemma~\ref{lemma11}, we express (\ref{45}) in the form
\beeq\label{47} \bigg|\int_{-\infty}^{\infty}e^{it\phi (\lambda
)}a(\lambda )\, d\lambda \bigg|\les t^{-1} \eneq where $\xi > 0>\xi
'$ are fixed, $\phi (\lambda ):=\lambda^2+\frac{\lambda}{t}(\xi -\xi
')$, and
$$a(\lambda )=\lambda\chi (\lambda )\chi_{[|\xi\lambda |>1,|\xi '\lambda |>1]}
(\langle\xi\rangle\langle\xi '\rangle )^{-1/2}\frac{m_+(\xi,\lambda )m_-(\xi ',\lambda )}{W(\lambda )}.$$
Let $\lambda_0=-\frac{\xi -\xi '}{2t}$. We have the bounds
\beeq\label{48} |a(\lambda )|\les(\langle\xi\rangle\langle\xi
'\rangle )^{\frac{-1}{2}}\chi (\lambda )\chi_{[|\xi\lambda |>1,|\xi
'\lambda |>1]}. \eneq By Corollary~\ref{cor7}, for small $|\lambda
|$
$$\bigg|\left(\frac{\lambda}{W(\lambda )}\right)'\bigg|\les\frac{1}{|\lambda |(\log |\lambda |)^2}$$
and by Lemma~\ref{lemma11}, for $|\xi\lambda |>1$, $|\xi '\lambda
|>1$,
$$|\partial_{\lambda}[m_+(\xi,\lambda )m_-(\xi ',\lambda )]|\les\lambda^{-2}(\xi^{-1}+|\xi '|^{-1}).$$
Hence, \beeq\label{49} |a'(\lambda
)|\les(\langle\xi\rangle\langle\xi '\rangle )^{-1/2}\chi (\lambda
)\chi_{[|\xi\lambda |>1,|\xi '\lambda |>1]}\left\{\frac{|\lambda
|^{-1}}{|\log\lambda |^2}+\lambda^{-2}(\xi^{-1}+|\xi
'|^{-1})\right\}. \eneq We will need to consider three cases in
order to prove (\ref{47}) via (\ref{46}), depending on where
$\lambda_0$ falls relative to the support of $a$.

\smallskip
\underline{Case 1:} $|\lambda_0 |\les 1$, $|\lambda_0|\gtr |\xi
|^{-1}+|\xi '|^{-1}$.

\smallskip
Note that the second inequality here implies that
$$\frac{\xi +|\xi '|}{t}\gtr\frac{\xi +|\xi '|}{\xi |\xi '|} \text{\ \ or\ \ }1\gtr\frac{t}{\xi |\xi '|}.$$
Furthermore, we remark that $a\equiv 0$ unless $\xi\gtr 1$ and $|\xi '|\gtr 1$.\\
Starting with the first integral on the right-hand side of
(\ref{46}) we conclude from (\ref{48}) that
$$\int\frac{|a(\lambda )|}{|\lambda -\lambda_0|^2+\delta^2}\,d\lambda\les(\langle\xi\rangle\langle\xi '\rangle )^{-1/2}t^{1/2}\les 1.$$
From the second integral we obtain from (\ref{49}) that
\begin{align*}
 \int\limits_{|\lambda -\lambda_0|>\delta}\frac{|a'(\lambda
)|}{|\lambda -\lambda_0|}\, d\lambda &\les
(\langle\xi\rangle\langle\xi '\rangle )^{-1/2}\delta^{-1}\int\frac{\chi (\lambda )\, d\lambda}{|\lambda |(\log |\lambda |)^2}\\
&\quad+(\langle\xi\rangle\langle\xi '\rangle
)^{-1/2}(\langle\xi\rangle^{-1}
+\langle\xi '\rangle^{-1})\;\delta^{-1}\!\!\!\!\!\!\!\int\limits_{\lambda >\xi^{-1}+|\xi '|^{-1}}\frac{d\lambda}{\lambda^2}\\
&\les\sqrt{\frac{t}{\langle\xi\rangle\langle\xi '\rangle}} \les 1.
\end{align*}

\smallskip
\underline{Case 2:} $|\lambda_0 |\les 1$, $|\lambda_0|\ll
\langle\xi\rangle^{-1}+\langle\xi '\rangle^{-1}$.

\smallskip
Then $|\lambda -\lambda_0|\sim |\lambda |$ on the support of $a$,
which implies that
\begin{align*}
\int\frac{|a(\lambda )|}{|\lambda -\lambda_0|^2+t^{-1}}\, d\lambda \les
 (\langle\xi\rangle\langle\xi '\rangle )^{-1/2}\int\limits_{\lambda >\xi^{-1}+|\xi '|^{-1}}\frac{d\lambda}{\lambda^2}
&\les\frac{\sqrt{\xi |\xi '|}}{\xi +|\xi '|} \les 1,
\end{align*}
and also
\begin{align*}
&\int\limits_{|\lambda -\lambda_0|>\delta}\frac{|a'(\lambda
)|}{|\lambda
-\lambda_0|}\, d\lambda \\
&\les(\langle\xi\rangle\langle\xi '\rangle
)^{-1/2}\Big(\int\limits_{\lambda
>\xi^{-1}+|\xi '|^{-1}} \frac{d\lambda}{\lambda^2(\log |\lambda
|)^2} +\int\limits_{\lambda
>\xi^{-1}+|\xi '|^{-1}}\frac{d\lambda}{\lambda^3}(\xi^{-1}+|\xi
'|^{-1})\Big) \\
&\les\frac{\sqrt{\xi |\xi '|}}{\xi +|\xi '|} \les 1.
\end{align*}

\smallskip
\underline{Case 3:} $|\lambda_0 |>> 1$,
$|\lambda_0|\gtr\xi^{-1}+|\xi ' |^{-1}$.

\smallskip
In this case, $|\lambda - \lambda_0 |\sim |\lambda_0|>>1$. Thus,
$$\int\frac{|a(\lambda )|}{|\lambda -\lambda_0|^2+t^{-1}}\, d\lambda
\les (\langle\xi\rangle\langle\xi '\rangle )^{-1/2}\frac{1}{\lambda_0^2+t^{-1}}\les 1$$
as well as, see (\ref{49}),
\begin{align*}
\int\limits_{|\lambda -\lambda_0|>\delta}\frac{|a'(\lambda
)|}{|\lambda -\lambda_0|}\; d\lambda &\les
(\langle\xi\rangle\langle\xi '\rangle )^{-1/2}\lambda_0^{-1}
\int\frac{\chi (\lambda )}{|\lambda |(\log |\lambda |)^2}\;d\lambda
\\
& +\int\frac{1}{\lambda^2}\chi_{[|\lambda |>\xi^{-1}+|\xi '|^{-1}]}
\frac{d\lambda}{\lambda_0}\:\frac{\xi +|\xi '|}{(\xi |\xi '|)^{3/2}}
\les 1,
\end{align*}
and \eqref{45} is proved.

\noindent Integrating by parts shows that \eqref{45'} is dominated
by
\[
(1+|t\pm(\xi-\xi')|)^{-1} \int (|a(\lambda)|+
|a'(\lambda)|)\,d\lambda \les (\la\xi\ra\la\xi'\ra)^{-\frac12}
(1+|t\pm(\xi-\xi')|)^{-1}
\] which is $\les t^{-\frac12}$
and the lemma is proved.
\end{proof}

Now we turn to the estimate of the oscillatory integral for the case
$|\xi\lambda|
>1$ and $|\xi '\lambda |<1$.

\begin{lemma}\label{lemma14} Let $\chi(\lambda;\xi,\xi')=\chi_{[|\xi\lambda| >1,|\xi '\lambda
|<1]}\chi(\lambda)$. Then \begin{align}\label{50} \sup_{\xi
>0>\xi '}\bigg|(\langle\xi\rangle\langle\xi '\rangle
)^{-1/2}\int_{-\infty}^{\infty}e^{it\lambda^2}\frac{\lambda\chi
(\lambda;\xi,\xi' )}{W(\lambda )}f_+(\xi,\lambda )f_-(\xi ',\lambda
)\,
d\lambda \bigg|&\les t^{-1} \\
\label{50'} \sup_{\xi
>0>\xi '}\bigg|(\langle\xi\rangle\langle\xi '\rangle
)^{-1/2}\int_{-\infty}^{\infty}e^{\pm it\lambda}\frac{\lambda\chi
(\lambda;\xi,\xi' )}{W(\lambda )}f_+(\xi,\lambda )f_-(\xi ',\lambda
)\, d\lambda \bigg|&\les t^{-\frac12}
\end{align}
 for all $t>0$ and similarly with
$\chi_{[|\xi\lambda |<1,|\xi '\lambda| > 1]}$.
\end{lemma}

\begin{proof}
As before, we write $f_+(\xi,\lambda
)=e^{i\xi\lambda}m_+(\xi,\lambda )$.  But because of $|\xi '\lambda
|<1$ we use the representation
$$f_-(\xi ',\lambda )=a_-(\lambda )u_0(\xi',\lambda )+b_-(\lambda )u_1(\xi',\lambda ).$$
In particular,
$$|f_-(\xi ',\lambda )|\les\sqrt{|\lambda |\langle\xi '\rangle}\big|\log |\lambda |\big|.$$
Moreover, from (\ref{34}) and (\ref{39}),
$$|\partial_{\lambda}f_-(\xi ',\lambda )|\les\langle\xi '\rangle^{1/2}|\lambda |^{-1/2}\big|\log |\lambda |\big|$$
provided $|\xi '\lambda |<1$. We apply (\ref{46}) with $\phi
(\lambda )=\lambda^2+\frac{\xi}{t}\lambda$ and
$$a(\lambda )=\frac{\lambda\chi (\lambda )}{W(\lambda )}(\langle\xi\rangle\langle\xi '\rangle )^{-1/2}\chi_{[|\xi\lambda| >1,|\xi '\lambda |<1]}m_+(\xi,\lambda )f_-(\xi ',\lambda ).$$
By the preceding, \beeq\label{51} |a(\lambda )|\les\frac{|\lambda
|^{1/2}}{\sqrt{\langle\xi\rangle}}\chi (\lambda )\chi_{[|\xi\lambda|
>1,|\xi '\lambda |<1]} \eneq and \beeq\label{52} |a'(\lambda )|\les
(|\lambda |\langle\xi\rangle )^{-1/2}\chi (\lambda
)\chi_{[|\xi\lambda| >1,|\xi '\lambda |<1]}.\eneq

\smallskip
\underline{Case 1:} $|\lambda_0|\les 1$, $|\xi\lambda_0|\gtr 1$.

\smallskip
Note in particular $|\xi |\gtr 1$. Here $\lambda_0=-\frac{\xi}{2t}$.
By (\ref{51}),
\begin{align*}
\int\frac{|a(\lambda )|}{|\lambda -\lambda_0|^2+t^{-1}}\,d\lambda &\les\langle\xi\rangle^{-1/2}\int\frac{\sqrt{|\lambda |}}{|\lambda -\lambda_0|^2+t^{-1}}\,d\lambda\\
&\les\langle\xi\rangle^{-1/2}|\lambda_0|^{1/2}\int\frac{d\lambda}{|\lambda -\lambda_0|^2+t^{-1}}+\langle\xi\rangle^{-1/2}\int\frac{|\lambda |^{1/2}}{|\lambda |^2+t^{-1}}\,d\lambda\\
&\les\langle\xi\rangle^{-1/2}t^{1/2}\left(\frac{\xi}{t}\right)^{1/2}+\langle\xi\rangle^{-1/2}t^{1/4}
\les 1
\end{align*}
Here we used that $|\xi\lambda_0|=\frac{\xi^2}{2t}\gtr 1$.\\
Next, write via (\ref{52}) \beeq\label{53} \int_{|\lambda
-\lambda_0|>\delta}\frac{|a'(\lambda )|}{|\lambda
-\lambda_0|}\,d\lambda\les\langle\xi\rangle^{-\frac{1}{2}}\int_{|\lambda
-\lambda_0|>\delta}\frac{1}{|\lambda |^{\frac{1}{2}}|\lambda
-\lambda_0|}\chi_{[|\xi\lambda| >1,|\xi '\lambda |<1]}\, d\lambda.
\eneq Distinguish the cases $\frac{1}{10}|\lambda |>|\lambda
-\lambda_0|$ and $\frac{1}{10}|\lambda |\leq |\lambda -\lambda_0|$
in the integral on the right-hand side.  This yields
\begin{align*}
(\ref{53})&\les\langle\xi\rangle^{-1/2}\int_{|\lambda
-\lambda_0|>\delta} \frac{d\lambda}{|\lambda
-\lambda_0|^{3/2}}+\langle\xi\rangle^{-1/2}
\int_{|\lambda |\les |\lambda_0|}\frac{d\lambda}{|\lambda |^{1/2}}|\lambda_0|^{-1}\\
&+\langle\xi\rangle^{-1/2}\int_{|\lambda |>|\lambda_0|}\frac{d\lambda}{|\lambda |^{3/2}}\\
&\les\langle\xi\rangle^{-1/2}\delta^{-1/2}+\langle\xi\rangle^{-1/2}|\lambda_0|^{-1/2}
\les \left(\frac{t}{\xi^2}\right)^{1/4}+|\xi\lambda_0|^{-1/2} \les
1.
\end{align*}

\smallskip
\underline{Case 2:} $|\lambda_0|\les 1$, $|\xi\lambda_0 |\ll 1$

\smallskip In that case, $|\lambda -\lambda_0|\sim|\lambda |$ on the
support of $a$.  Consequently,
$$\int\frac{|a(\lambda )|}{|\lambda -\lambda_0|^2+t^{-1}}\,d\lambda\les\langle\xi\rangle^{-\frac{1}{2}}\int_{|\xi |^{-1}}^{\infty}|\lambda |^{-\frac{3}{2}}\, d\lambda\les 1.$$
Moreover,
$$\int_{|\lambda -\lambda_0|>\delta}\frac{|a'(\lambda )|}
{|\lambda -\lambda_0|}\,d\lambda\les\int_{|\xi |^{-1}}^{\infty}
\frac{(|\lambda |\langle\xi\rangle )^{-\frac{1}{2}}}{|\lambda
|}\,d\lambda \les 1.$$

\smallskip
\underline{Case 3:} $|\lambda_0|>>1$.

\smallskip
In that case, $|\lambda -\lambda_0|\sim |\lambda_0|$ on $\supp (a)$.
Since $|a(\lambda )|\les 1$ by (\ref{51}), it follows that
$$\int\frac{|a(\lambda )|}{|\lambda -\lambda_0|^2+t^{-1}}\,d\lambda\les 1.$$
Similarly, since $|a'(\lambda )|\les (\xi |\lambda
|)^{-\frac{1}{2}}$, it follows that
$$\int_{|\lambda -\lambda_0|>\delta}\frac{|a'(\lambda )|}{|\lambda -\lambda_0|}\,d\lambda\les\int\frac{(|\lambda |\langle\xi\rangle )^{-\frac{1}{2}}}{|\lambda_0|}\chi (\lambda )\, d\lambda\les 1.$$
This proves (\ref{50}).

To prove \eqref{50'}, we integrate by parts to obtain the upper
bound
\[
(1+|t\pm\xi|)^{-1}\int (|a(\lambda)|+|a'(\lambda)|)\,d\lambda \les
(1+|t\pm\xi|)^{-1} \xi^{-\frac12}\les t^{-\frac12}
\]
and the lemma is proved. The other case $\chi_{[|\xi\lambda |<1,|\xi
'\lambda| > 1]}$ is treated in an analogous fashion.
\end{proof}

The remaining cases for the small energy part of
(\ref{eq:schr_oscill}) are $\xi
>\xi '>|\lambda |^{-1}$ and $\xi '<\xi <-|\lambda |^{-1}$.  By
symmetry it will suffice to treat the former case.  As usual, we
need to consider reflection and transmission coefficients, therefore
we write \beeq\label{54} f_-(\xi,\lambda )=\alpha_-(\lambda
)f_+(\xi,\lambda )+\beta_-(\lambda )\overline{f_+(\xi,\lambda )}.
\eneq Then, with $W(\lambda )=W(f_+(\cdot,\lambda
),f_-(\cdot,\lambda ))$,
$$W(\lambda )=\beta_-(\lambda )W(f_+(\cdot,\lambda ),\overline{f_+(\cdot,\lambda )})=-2i\lambda\beta_-(\lambda )$$
and
\begin{align*}
W(f_-(\cdot,\lambda ),\overline{f_+(\cdot,\lambda )})&=\alpha_-(\lambda )W(f_+(\cdot,\lambda ),\overline{f_+(\cdot,\lambda )})\\
&=-2i\lambda\alpha_-(\lambda ).
\end{align*}
Thus, when $\lambda >0$ is small, \beeq\label{55} \beta_-(\lambda
)=i\left(1+ic_3+i\frac{2}{\pi}\log\lambda \right)+O(|\lambda
|^{\frac{1}{2}-\eps}) \eneq and
\begin{align}\label{56}
\alpha_-(\lambda )&=\frac{1}{-2i\lambda}W\left(a_+(\lambda )u_0(\cdot,\lambda )-b_+(\lambda )u_1(\cdot,\lambda ),\overline{a_+}(\lambda )u_0(\cdot,\lambda )+\overline{b_+}(\lambda )u_1(\cdot,\lambda )\right)\nonumber\\
&=\frac{1}{-2i\lambda}\left(a_+\overline{b_+}(\lambda )+\overline{a_+}(\lambda )b_+(\lambda )\right)\nonumber\\
&=\frac{i}{\lambda}\text{Re}(a_+\overline{b_+}(\lambda))\nonumber\\
&=\frac{i}{\lambda}\text{Re}\left(-i|c_0|^2c_1\lambda(1+ic_1\log\lambda +ic_3)+O(\lambda^{\frac{3}{2}-\eps})\right)\nonumber\\
&=i\left(\frac{2}{\pi}\log\lambda
+c_3\right)+O(\lambda^{\frac{1}{2}-\eps}).
\end{align}
In passing, we remark that $1+|\alpha_-|^2=|\beta_-|^2$.  Finally,
it follows from Corollary~\ref{cor6} that the $O$-terms can be
differentiated once in $\lambda$; they then become
$O(\lambda^{-\frac{1}{2}-\eps})$, $\eps >0$ arbitrary.

\begin{lemma}\label{lemma15}
For any $t>0$ \begin{align}\label{57} \sup_{\xi >\xi
'>0}\bigg|(\langle\xi\rangle\langle\xi '\rangle )^{-\frac{1}{2}}\int
e^{it\lambda^2}\frac{\lambda\chi (\lambda )}{W(\lambda )}\chi_{[\xi
'\lambda >1]}f_+(\xi,\lambda )f_-(\xi ',\lambda )\, d\lambda
\bigg|&\les t^{-1}\\
\label{57'} \sup_{\xi >\xi '>0}\bigg|(\langle\xi\rangle\langle\xi
'\rangle )^{-\frac{1}{2}}\int e^{\pm it\lambda}\frac{\lambda\chi
(\lambda )}{W(\lambda )}\chi_{[\xi '\lambda >1]}f_+(\xi,\lambda
)f_-(\xi ',\lambda )\, d\lambda \bigg|&\les t^{-\frac12}
 \end{align} and similarly for $\sup_{\xi ' <\xi
<0}$ and $\chi_{[|\xi\lambda |>1]}$.
\end{lemma}

\begin{proof}
Using (\ref{54}), we reduce (\ref{57}) to two estimates:
\begin{equation}\label{58}
\sup_{\xi >\xi '>0}\bigg|\int e^{it\lambda^2}e^{i\lambda (\xi +\xi
')}\frac{\lambda\chi (\lambda )}{W(\lambda )} \frac{\chi_{[\xi
'|\lambda |>1]}}{\sqrt{\langle\xi\rangle\langle\xi '\rangle
}}m_+(\xi,\lambda )m_+(\xi ',\lambda )\alpha_-(\lambda )\, d\lambda
\bigg| \les t^{-1}
\end{equation}
and
\begin{align}\label{59}
\sup_{\xi >\xi '>0}&\bigg| \int e^{it\lambda^2}e^{i\lambda (\xi -\xi
')}\frac{\lambda\chi (\lambda )} {W(\lambda )}\frac{\chi_{[\xi
'|\lambda |>1]}}{\sqrt{\langle\xi\rangle\langle\xi' \rangle
}}m_+(\xi,\lambda )\overline{m_+(\xi ',\lambda )}\beta_-(\lambda )\,
d\lambda \bigg| \les t^{-1}
\end{align}
We apply (\ref{46}) to (\ref{58}) with fixed $\xi >\xi '>0$ and
\begin{align*}
\phi (\lambda )&=\lambda^2+\frac{\lambda}{t}(\xi +\xi '),\\
a(\lambda )&=(\langle\xi\rangle\langle\xi '\rangle
)^{-\frac{1}{2}}\frac{\lambda\chi (\lambda )}{W(\lambda )}\chi_{[\xi
'|\lambda |>1]}\alpha_-(\lambda )m_+(\xi,\lambda )m_+(\xi ',\lambda
).
\end{align*}
Then from (\ref{56}), \beeq\label{60} |a(\lambda
)|\les(\langle\xi\rangle\langle\xi '\rangle )^{-\frac{1}{2}}\chi
(\lambda )\chi_{[\xi '|\lambda |>1]} \eneq and from our derivative
bounds on $W$, $\alpha_-$, and $m_+(\xi,\lambda)$, see (\ref{41})
for the latter, we conclude that \beeq\label{61} |a'(\lambda )|\les
|\lambda |^{-1}(\langle\xi\rangle\langle\xi '\rangle
)^{-\frac{1}{2}}\chi (\lambda )\chi_{[\xi '|\lambda |>1]}. \eneq
This bound will suffice for the Schr\"odinger evolution. For the
wave evolution, we also need an integrable estimate on
$|a'(\lambda)|$. It is
\[
|a'(\lambda)|\les |\lambda |^{-1}(\langle\xi\rangle\langle\xi
'\rangle )^{-\frac{1}{2}}\chi (\lambda )\chi_{[\xi '|\lambda |>1]}
\Big( |\log\lambda|^{-2}+ |\lambda\xi'|^{-1}\Big)
\]
which one obtains by combining \eqref{56} with our asymptotic bound
for $\frac{\lambda}{W(\lambda)}$ above.
\smallskip
\underline{Case 1:} Suppose $|\lambda_0|\les 1$ and $|\xi '\lambda_0
|>1$, where $\lambda_0=-\frac{\xi +\xi '}{2t}$.  Note $\xi >\xi
'\gtr 1$.

\smallskip Then
\[
\int\frac{|a(\lambda )|}{|\lambda -\lambda_0|^2+t^{-1}}\,d\lambda
\les (\langle\xi\rangle\langle\xi '\rangle
)^{-\frac{1}{2}}\int\frac{d\lambda}{|\lambda -\lambda_0|^2+t^{-1}}
\les\sqrt{\frac{t}{\xi\xi '}}\les 1
\]
since $|\xi '\lambda_0|\sim\frac{\xi\xi '}{t}>1$. As for the
derivative term in (\ref{46}), we infer from (\ref{61}) that
\beeq\label{62} \int_{|\lambda -\lambda_0|>\delta}\frac{|a'(\lambda
)|}{|\lambda -\lambda_0 |}\,d\lambda\les(\langle\xi\rangle\langle\xi
'\rangle )^{-\frac{1}{2}}\int_{|\lambda
-\lambda_0|>\delta}\frac{d\lambda}{|\lambda ||\lambda -\lambda_0
|}\chi_{[|\lambda\xi '|>1]} \eneq Again, we need to distinguish
between $|\lambda -\lambda_0|>\frac{1}{10}|\lambda_0|$ and $|\lambda
-\lambda_0|<\frac{1}{10}|\lambda_0|$.  Thus, since $\xi\xi '>t$,
\begin{align*}
(\ref{62})&\les (\langle\xi\rangle\langle\xi '\rangle )^{-\frac{1}{2}}
\int_{1/\xi '}^{\infty}\frac{d\lambda}{\lambda^2}+(\langle\xi\rangle\langle\xi '\rangle )^{-1/2}
|\lambda_0|^{-1}\log \left(t^{1/2}|\lambda_0|\right)\\
&\les 1+\frac{t^{\frac{1}{2}}}{\xi}\log
\left(\frac{\xi}{t^{1/2}}\right) \les 1
\end{align*}
since also $\xi^2>t$.

\smallskip
\underline{Case 2:} $|\lambda_0|\les 1$, $|\lambda_0|\ll
\frac{1}{\xi '}$.

\smallskip Then $|\lambda -\lambda_0|\sim|\lambda |$ on the support
of $a(\lambda )$.  Hence,
$$\int\frac{|a(\lambda )|}{|\lambda -\lambda_0|^2+t^{-1}}\,d\lambda\les
(\langle\xi\rangle\langle\xi '\rangle )^{-\frac{1}{2}}\int_{1/\xi '}^{\infty}
\frac{d\lambda}{\lambda^2}\les\sqrt{\frac{\xi '}{\langle\xi\rangle}}<1$$
and
$$\int_{|\lambda -\lambda_0|>\delta}\frac{|a'(\lambda )|}{|\lambda -\lambda_0|}\,d\lambda
\les (\langle\xi\rangle\langle\xi '\rangle )^{-\frac{1}{2}}\int_{1/\xi '}^{\infty}\frac{d\lambda}{\lambda^2}<1.$$

\smallskip
\underline{Case 3:} $|\lambda_0|>>1$, $|\lambda_0|\gtr\frac{1}{\xi
'}$.

\smallskip
Then $|\lambda -\lambda_0|\sim |\lambda_0|$ on $\supp (a)$.
Therefore, $|a (\lambda )|\les 1$ implies that
$$\int\frac{|a(\lambda )|}{|\lambda -\lambda_0|^2+t^{-1}}\,d\lambda\les 1$$
and
\begin{align*}
\int_{|\lambda -\lambda_0|>\delta}\frac{|a'(\lambda )|} {|\lambda
-\lambda_0|}\,d\lambda&\les (\langle\xi\rangle\langle\xi '\rangle
)^{-\frac{1}{2}} |\lambda_0|^{-1}\int_{\frac{1}{\langle\xi
'\rangle}}^{1}\frac{d\lambda}{|\lambda |}
\les\frac{1}{|\lambda_0|}\frac{1}{\langle\xi '\rangle}\log\langle\xi
'\rangle \les 1.
\end{align*}
This concludes the proof of (\ref{58}). (\ref{59}) is completely
analogous and~\eqref{57} follows.

As usual, integration by parts proves that \eqref{57'}  is dominated
by
\[
(1+|t\pm (\xi \pm \xi')|)^{-1} \int
(|a(\lambda)|+|a'(\lambda)|)\,d\lambda \les
(\langle\xi\rangle\langle\xi '\rangle )^{-\frac{1}{2}} (1+|t\pm (\xi
\pm \xi')|)^{-1}
\] which is $\les t^{-\frac12}$.

Finally, the  case of $\xi '<\xi <0$, $|\xi\lambda |>1$ follows from
the case considered in this proof by a reflection around $\xi=0$.
\end{proof}

We are done with the contributions of small $\lambda$ to the
oscillatory integral~(\ref{eq:schr_oscill})
and~\eqref{eq:wave_oscill}. To conclude the proof of
\eqref{eq:schr_d} for $d=1$ it suffices to prove the following
statement. The wave equation will be treated separately, see
Lemma~\ref{lem:wave_crux}.

\begin{lemma}\label{lemma16}
For all $t>0$, \beeq\label{63} \sup_{\xi >\xi
'}\bigg|(\langle\xi\rangle\langle\xi '\rangle
)^{-\frac{1}{2}}\int_{-\infty}^{\infty}e^{it\lambda^2}\frac{\lambda
(1-\chi )(\lambda )}{W(\lambda )}f_+(\xi,\lambda )f_-(\xi ',\lambda
)\, d\lambda \bigg|\les t^{-1}. \eneq
\end{lemma}

\begin{proof}
We observed above, see (\ref{54}), that $W(\lambda
)=-2i\lambda\beta_-(\lambda )$. Since $|\beta_-(\lambda )|\geq 1$,
this implies that $|W(\lambda )|\geq 2|\lambda |$. In particular,
$W(\lambda )\neq 0$ for every $\lambda\neq 0$. In order to prove
(\ref{63}), we will need to distinguish the cases $\xi >0>\xi '$,
$\xi >\xi '>0$, and $0>\xi >\xi '$.  By symmetry, it will suffice to
consider the first two.

\smallskip
\underline{Case 1:} $\xi >0>\xi '$.

\smallskip
In this case we need to prove that \beeq\label{64} \sup_{\xi >0>\xi
'}\bigg|(\langle\xi\rangle\langle\xi '\rangle )^{-\frac{1}{2}}\int
e^{it[\lambda^2+\frac{\xi -\xi '}{t}\lambda]}\frac{\lambda (1-\chi
)(\lambda )}{W(\lambda )}m_+(\xi,\lambda )m_-(\xi ',\lambda )\,
d\lambda \bigg|\les t^{-1}. \eneq Apply (\ref{46}) with $\phi
(\lambda )=\lambda^2+\frac{\xi -\xi '}{t}\lambda$ and
$$a(\lambda )= (\langle\xi\rangle\langle\xi '\rangle )^{-\frac{1}{2}}
\frac{\lambda (1-\chi )(\lambda )}{W(\lambda )}m_+(\xi,\lambda )m_-(\xi ',\lambda ).$$
Hence, with $\lambda_0=-\frac{\xi -\xi '}{2t}$,
\begin{align*}
(\ref{64})&\les t^{-1}\left(\int\frac{|a(\lambda )|}{|\lambda
-\lambda_0|^2+t^{-1}}\,d\lambda
+\int\limits_{|\lambda -\lambda_0|>\delta}\frac{|a'(\lambda )|}{|\lambda -\lambda_0|}\,d\lambda\right)\\
&=:t^{-1}(A+B).
\end{align*}
If $|\lambda_0|\ll 1$, then
$$A\les ||a||_{\infty}\les 1.$$
On the other hand, if $|\lambda_0|\gtr 1$, then $\xi +|\xi '|\gtr t$
so that
$$A\les t^{\frac{1}{2}}||a||_{\infty}\les t^{\frac{1}{2}}(\langle\xi\rangle\langle\xi '\rangle )^{-\frac{1}{2}}
\les\sqrt{\frac{t}{\langle\xi\rangle\langle\xi '\rangle}}\les 1.$$
Here we used that
$$\sup_{\xi}\sup_{|\lambda |\gtr 1}|m_{\pm}(\xi,\lambda )|\les 1$$
which follows from the fact that \beeq\label{65} m_+(\xi,\lambda
)=1+\int_{\xi}^{\infty}\frac{1-e^{-2i(\tilde{\xi} -\xi
)\lambda}}{2i\lambda}V(\tilde{\xi} )m_+(\tilde{\xi},\lambda
)d\tilde{\xi} \eneq with
$V(\tilde{\xi}=O(\langle\tilde{\xi}\rangle^{-2})$.  Moreover, from
our assumptions on $r(x)$ we recall that
$$\bigg|\frac{d^\ell}{d\xi^\ell}V(\xi )\bigg|\les\langle\xi\rangle^{-2-\ell},\quad\forall\; \ell\geq 0.$$
We shall need these bounds to estimate $B$ above. From (\ref{65}),
for $\xi\geq 0$
$$m_+(\xi,\lambda )=1+O(\lambda^{-1}\langle\xi\rangle^{-1})$$
as well as for $\xi\geq 0$
\begin{align}\label{66}
\partial^j_{\xi}m_+(\xi,\lambda )&=O(\lambda^{-1}\langle\xi\rangle^{-1-j}),\qquad j=1,2
\end{align}
\begin{align}\label{67}
\partial_{\lambda}m_+(\xi,\lambda )&=O(\lambda^{-2}\langle\xi\rangle^{-1})
\end{align}
\begin{align}\label{68}
\partial_{\lambda}\partial_{\xi}m_+(\xi,\lambda )&=O(\lambda^{-2}\langle\xi\rangle^{-2})
\end{align}
To verify (\ref{66}), one checks that
\begin{align}\label{69}
\partial_{\xi}m_+(\xi,\lambda )&=
\frac{1}{2i\lambda}\int_{\xi}^{\infty}[1-e^{2i(\xi
-\tilde{\xi})\lambda}]V'(\tilde{\xi}) m_+(\tilde{\xi},\lambda )
\, d\tilde{\xi}\\
&+\frac{1}{2i\lambda}\int_{\xi}^{\infty}[1-e^{2i(\xi
-\tilde{\xi})\lambda}]
V(\tilde{\xi})\partial_{\tilde{\xi}}m_+(\tilde{\xi},\lambda )\,
d\tilde{\xi}.\nonumber
\end{align}
By our estimates on $V$, the integral on the right-hand side of
(\ref{69}) is $O(\lambda^{-1}\langle\xi\rangle^{-2})$ and (\ref{66})
follows for $j=1$. For $j=2$ note that
\begin{align}\nonumber
\partial_{\xi}^2m_+(\xi,\lambda )&=
\frac{1}{2i\lambda}\int_{\xi}^{\infty}[1-e^{2i(\xi
-\tilde{\xi})\lambda}]\, V''(\tilde{\xi}) m_+(\tilde{\xi},\lambda )
\, d\tilde{\xi}\\
& + \frac{1}{i\lambda}\int_{\xi}^{\infty}[1-e^{2i(\xi
-\tilde{\xi})\lambda}]\, V'(\tilde{\xi})
\partial_{\tilde\xi} m_+(\tilde{\xi},\lambda )
\, d\tilde{\xi}\nonumber\\
&+\frac{1}{2i\lambda}\int_{\xi}^{\infty}[1-e^{2i(\xi
-\tilde{\xi})\lambda}]
V(\tilde{\xi})\partial^2_{\tilde{\xi}}m_+(\tilde{\xi},\lambda )\,
d\tilde{\xi}.\nonumber
\end{align}
which again implies the desired bound. For (\ref{67}) we compute
\begin{align*}
\partial_{\lambda}m_+(\xi,\lambda )=-&\int_{\xi}^{\infty}\frac{1-e^{2i(\xi -\tilde{\xi})\lambda}}
{2i\lambda^2}V(\tilde{\xi})m_+(\tilde{\xi},\lambda )\, d\tilde{\xi}\\
&+\frac{1}{2i\lambda^2}\int_{\xi}^{\infty}e^{2i(\xi -\tilde{\xi})\lambda}\partial_{\tilde{\xi}}
\left[(\xi -\tilde{\xi})V(\tilde{\xi})m_+(\tilde{\xi},\lambda )\right]\, d\tilde{\xi}\\
&+\int_{\xi}^{\infty}\frac{1-e^{2i(\xi
-\tilde{\xi})\lambda}}{2i\lambda}V(\tilde{\xi}
)\partial_{\lambda}m_+(\tilde{\xi},\lambda )\, d\tilde{\xi}
\end{align*}
so that
$$\partial_{\lambda}m_+(\xi,\lambda )=O(\lambda^{-2}\langle\xi\rangle^{-1})$$
as claimed. Finally, compute
\begin{align*}
\partial^2_{\xi\lambda}m_+(\xi,\lambda )&=\frac{1}{\lambda}\int_{\xi}^{\infty}e^{2i(\xi -\tilde{\xi})
\lambda}V(\tilde{\xi})m_+(\tilde{\xi},\lambda )\, d\tilde{\xi}\\
&+\frac{1}{2i\lambda^2} V(\xi )m_+(\xi,\lambda )+
\frac{1}{2i\lambda}\int_{\xi}^{\infty}e^{2i(\xi -\tilde{\xi})\lambda}\partial_{\tilde{\xi}}
[(\xi -\tilde{\xi})Vm_+(\tilde{\xi},\lambda )]\, d\tilde{\xi}\\
&+\frac{1}{2i\lambda^2}\int_{\xi}^{\infty}e^{2i(\xi
-\tilde{\xi})\lambda}
\partial_{\tilde{\xi}}[V(\tilde{\xi})m_+(\tilde{\xi},\lambda )]\, d\tilde{\xi}\\
&-\int_{\xi}^{\infty}e^{2i(\xi
-\tilde{\xi})\lambda}V(\tilde{\xi})\partial_{\lambda}m_+(\tilde{\xi},\lambda
)\, d\tilde{\xi}
\end{align*}
Integrating by parts in the first and third terms, and using the
previous bounds,  yields the desired estimate. As a corollary, we
obtain (take $\xi =0$)
\begin{align*}
W(\lambda )&=W(f_+(\cdot,\lambda ),f_-(\cdot,\lambda ))\\
&=m_+(\xi,\lambda )[m'_-(\xi,\lambda )-i\lambda m_-(\xi,\lambda )]-m_-(\xi,\lambda )
[m'_+(\xi,\lambda )+i\lambda m_+(\xi,\lambda )]\\
&=-2i\lambda (1+O(\lambda^{-1}))+O(\lambda^{-1}) =-2i\lambda+O(1)
\end{align*}
with derivatives $W'(\lambda )=-2i+O(\lambda^{-1})$ as $|\lambda
|\to\infty$.

\noindent  Next, we estimate $B$. First, we conclude from our bounds
on $W(\lambda )$ and $m_+(\xi,\lambda )$ as well as $m_-(\xi
',\lambda )$ that
$$|a'(\lambda )|\les (\langle\xi\rangle\langle\xi '\rangle )^{-\frac{1}{2}}\chi_{[|\lambda |\gtr 1]}|\lambda |^{-2}.$$
Let us first consider the case where $|\lambda_0|\gtr 1$.  Then
\begin{align*}
B&\les(\langle\xi\rangle\langle\xi '\rangle )^{-1/2}
\int_{\binom{|\lambda -\lambda_0|>\delta}{|\lambda |\gtr 1}}\frac{d\lambda}{|\lambda |^2|\lambda -\lambda_0|}\\
&\les (\langle\xi\rangle\langle\xi '\rangle
)^{-1/2}\left\{\int_1^{\infty}\frac{d\lambda}
{\lambda^3}+\frac{1}{|\lambda_0|^2}\int_{\frac{|\lambda_0|}{5}>|\lambda
-\lambda_0 |>\delta}
\frac{d\lambda}{|\lambda -\lambda_0|}\right\}\\
&\les 1+\sqrt{\frac{t}{\langle\xi\rangle\langle\xi
'\rangle}}\frac{1}{|\lambda_0|t^{1/2}}\log_+(\lambda_0t^{1/2}) \les
1
\end{align*}
Here we used that $\frac{t}{\langle\xi\rangle\langle\xi
'\rangle}\les 1$ which follows from $|\lambda_0|\gtr 1$. If
$|\lambda_0|\ll 1$, then $|\lambda-\lambda_0|\sim|\lambda |$ on the
support of $a$; thus $B\les 1$ trivially. This finishes the case
$\xi >0>\xi '$.

\smallskip
\underline{Case 2:} To deal with the case $\xi >\xi '>0$, we use
(\ref{54}).  Thus,
$$f_-(\xi ',\lambda )=\alpha_-(\lambda )f_+(\xi ',\lambda )+\beta_-(\lambda )\overline{f_+(\xi ',\lambda )}$$
where
\begin{align*}
\alpha_-(\lambda )&=\frac{W(f_-(\cdot,\lambda ),\overline{f_+(\cdot,\lambda )})}{-2i\lambda}\\
\beta_-(\lambda )&=\frac{W(f_+(\cdot,\lambda ),f_-(\cdot,\lambda
))}{-2i\lambda}=\frac{W(\lambda )}{-2i\lambda}
\end{align*}
From our large $\lambda$ asymptotics of $W(\lambda )$ we deduce that
\begin{equation}\label{eq:beta-}\beta_-(\lambda )=1+O(\lambda ^{-1}), \quad \beta_-'(\lambda
)=O(\lambda^{-2}).
\end{equation} For $\alpha_-(\lambda )$ we calculate, again at
$\xi =0$,
\begin{align*}
W(f_-(\cdot,\lambda ),\overline{f_+(\cdot,\lambda
)})=\,&m_-(\xi,\lambda )(\overline{m}_+'(\xi,\lambda )
-2i\lambda\overline{m}_+(\xi,\lambda ))\\
&-\overline{m}_+(\xi,\lambda )(m'_-(\xi,\lambda )-2i\lambda m_-(\xi,\lambda ))\\
=\,&m_-(\xi,\lambda )\overline{m}'_+(\xi,\lambda)-m_-'(\xi,\lambda )\overline{m}_+(\xi,\lambda )\\
=\,&O(\lambda^{-1})
\end{align*}
so that
\begin{equation}\label{eq:alpha-}\alpha_-(\lambda )=O(\lambda^{-2}),\quad \alpha'_-(\lambda )=O(\lambda^{-3}).
\end{equation}
Thus, we are left with proving the two bounds
\begin{align}\label{70}
\sup_{\xi >\xi
'>0}&\bigg|\int_{-\infty}^{\infty}e^{it\lambda^2}e^{i\lambda (\xi
+\xi ')}\frac{\lambda (1-\chi (\lambda ))}{W(\lambda
)}\alpha_-(\lambda )\frac{m_+(\xi,\lambda )m_+(\xi ',\lambda
)}{\sqrt{\langle\xi\rangle\langle\xi '\rangle }}\, d\lambda \bigg|
\les t^{-1} \\
\label{71} \sup_{\xi >\xi
'>0}&\bigg|\int_{-\infty}^{\infty}e^{it\lambda^2}e^{i\lambda (\xi
-\xi ')}\frac{\lambda (1-\chi (\lambda ))}{W(\lambda
)}\beta_-(\lambda )\frac{m_+(\xi,\lambda )\overline{m_+(\xi
',\lambda )}}{\sqrt{\langle\xi\rangle\langle\xi '\rangle }}\,
d\lambda \bigg| \les t^{-1}
\end{align}
for any $t>0$. This, however, follows by means of the exact same
arguments which we use to prove (\ref{64}).  Note that in (\ref{70})
the critical point of the phase is
$$\lambda_0=-\frac{\xi +\xi '}{2t}$$
whereas in (\ref{71}) it is $\lambda_0=-\frac{\xi -\xi '}{2t}.$ In
either case it follows from $|\lambda_0|\gtr 1$ that $\xi\gtr t$.
Hence we can indeed argue as in Case~1. This finishes the proof of
the lemma, and thus also of Theorem~\ref{thm1}.
\end{proof}

Now for the wave case. We will tacitly use some elements of the
previous proof.

\begin{lemma}\label{lem:wave_crux}
For all $t>0$, \begin{align}\label{eq:wave_biglam} &
\bigg|\int_{-\infty}^\xi (\langle\xi\rangle\langle\xi '\rangle
)^{-\frac{1}{2}}\int_{-\infty}^{\infty}e^{\pm
it\lambda}\frac{\lambda (1-\chi )(\lambda )}{W(\lambda
)}f_+(\xi,\lambda )f_-(\xi ',\lambda
)\, d\lambda\;  \phi(\xi')\, d\xi'\bigg| \nonumber\\
& \les t^{-\frac12} \int \big(|\phi(\xi')|+|\phi'(\xi')|\big)\, d\xi'.
\end{align}
with a constant that does not depend on $\xi$.
\end{lemma}

\begin{proof}
In order to prove (\ref{eq:wave_biglam}), we will need to
distinguish the cases $\xi >0>\xi '$, $\xi >\xi '>0$, and $0>\xi
>\xi '$.  By symmetry, it will suffice to consider the first two.

\smallskip
\underline{Case 1:} $\xi >0>\xi '$.

\smallskip
Integrating by parts yields
 \begin{align*} &\bigg|(\langle\xi\rangle\langle\xi '\rangle
)^{-\frac{1}{2}}\int e^{i\lambda(\pm t+\xi -\xi ')}\frac{\lambda
(1-\chi )(\lambda )}{W(\lambda )}m_+(\xi,\lambda )m_-(\xi ',\lambda
)\, d\lambda \bigg|\\
&\les (\langle\xi\rangle\langle\xi '\rangle )^{-\frac{1}{2}}
|t\pm(\xi-\xi')|^{-1} \les t^{-\frac12}
\end{align*}
provided $|t\pm(\xi-\xi')|\ge1$. If this fails, then we need to
integrate by parts in~$\xi'$ to remove one factor of~$\lambda$:
since $\lambda e^{-i\xi'\lambda} = i\partial_{\xi'}
e^{-i\xi'\lambda}$, it follows that
\begin{align*}
  & \int_{-\infty}^\xi \langle\xi\rangle^{-\frac12} \langle\xi '\rangle^{-\frac{1}{2}}
  \int e^{i\lambda(\pm t+\xi -\xi ')}\frac{\lambda
(1-\chi )(\lambda )}{W(\lambda )}m_+(\xi,\lambda )m_-(\xi ',\lambda
)\, d\lambda\, \phi(\xi')\,d\xi' =\\
& i\la \xi\ra^{-1} \int e^{\pm it\lambda} \frac{(1-\chi )(\lambda
)}{W(\lambda )}m_+(\xi,\lambda )m_-(\xi,\lambda )\, d\lambda\,
\phi(\xi)
\\
 &-i\int_{-\infty}^\xi \langle\xi\rangle^{-\frac12}
  \int e^{i\lambda(\pm t+\xi -\xi ')}\frac{
(1-\chi )(\lambda )}{W(\lambda )}m_+(\xi,\lambda )
\partial_{\xi'}\big[\langle\xi '\rangle^{-\frac{1}{2}}m_-(\xi ',\lambda )\, \phi(\xi')\big]\,
d\lambda\,d\xi'
\end{align*}
Denote the two expressions after the equality sign by $A$ and $B$, respectively.
First, exploiting the cancelation due to $W(-\lambda)=-W(\lambda)+ O(1)$ as $\lambda\to\infty$, we see that
\[
\sup_{\xi>0>\xi'}\Big|\int e^{it\lambda} \frac{(1-\chi )(\lambda
)}{W(\lambda )}m_+(\xi,\lambda )m_-(\xi ,\lambda )\, d\lambda\Big|
\les 1
\]
Furthermore, since $\vert\partial_{\lambda}\{\frac{(1-\chi )(\lambda
)}{W(\lambda )}m_+(\xi,\lambda )m_-(\xi ,\lambda )\}\vert\les \chi_{[|\lambda |\gtr 1]}|\lambda |^{-2}$, integrating by parts in $\lambda$ shows that the left-hand side is in fact $\les t^{-1}$.
Hence,
\[
A\les \la t\ra^{-1} \sup|\phi|\le \la t\ra^{-1} \int ( |\phi'(\xi')|+|\phi(\xi')|)\, d\xi'.
\]
Second, by the same cancelation,
\begin{align*}
B\les \int (\la\xi\ra\la\xi'\ra)^{-\frac12} (1+|t\pm (\xi-\xi')|)^{-1}( |\phi'(\xi')|+|\phi(\xi')|)\, d\xi'\\\les  \la t\ra^{-1} \int ( |\phi'(\xi')|+|\phi(\xi')|)\, d\xi'.
\end{align*}
which gives the desired bound as usual.

\smallskip
\underline{Case 2:}  $\xi >\xi '>0$

\medskip

In analogy with \eqref{70} and \eqref{71} we need to consider
\begin{align}\label{70'}
& \int_{-\infty}^{\infty}e^{it\lambda}e^{i\lambda (\xi
+\xi ')}\frac{\lambda (1-\chi (\lambda ))}{W(\lambda
)}\alpha_-(\lambda )\frac{m_+(\xi,\lambda )m_+(\xi ',\lambda
)}{\sqrt{\langle\xi\rangle\langle\xi '\rangle }}\, d\lambda , \\
\label{71'} &\int_{-\infty}^{\infty}e^{it\lambda}e^{i\lambda (\xi
-\xi ')}\frac{\lambda (1-\chi (\lambda ))}{W(\lambda
)}\beta_-(\lambda )\frac{m_+(\xi,\lambda )\overline{m_+(\xi
',\lambda )}}{\sqrt{\langle\xi\rangle\langle\xi '\rangle }}\,
d\lambda .
\end{align}
The integral in \eqref{70'} is $\les \la t\ra^{-\frac12}$ uniformly in $\xi,\xi'$ due to the decay of~$\alpha_-$, see~\eqref{eq:alpha-}.
On the other hand, the integral in~\eqref{71'} is not a bounded function in $\xi,\xi'$ due to the lack of decay in~$\lambda$, see~\eqref{eq:beta-}.
Thus, we again need to redeem one power of $\lambda$ via a $\xi'$ differentiation, see above.
\end{proof}

\end{document}